\documentclass{amsart}
\usepackage[T1]{fontenc}
\usepackage[utf8]{inputenc}

\usepackage{lmodern}
\usepackage{amsmath}
\usepackage{amssymb}
\usepackage{stmaryrd}
\usepackage{mathrsfs}
\usepackage{tikz}
\usepackage{empheq}
\usepackage{bbm}
\definecolor{theme}{RGB}{15, 87, 24} 
\definecolor{lighttheme}{RGB}{51, 153, 63} 
\definecolor{block}{RGB}{102, 60, 0} 
\definecolor{lightblock}{RGB}{255, 238, 214} 
\definecolor{alert}{RGB}{212, 43, 43} 

\usepackage{hyperref}
\hypersetup{colorlinks=true,
            linkcolor=theme,
            citecolor=lighttheme}

\usepackage{geometry}
\usepackage{amsthm}
\newtheorem{thm}{Theorem}[section]
\newtheorem{prop}[thm]{Proposition}
\newtheorem{conj}{Conjecture}
\newtheorem{question}[conj]{Question}
\newtheorem{lemma}[thm]{Lemma}
\newtheorem{deft}[thm]{Definition}

\newtheorem{fact}[conj]{Fact}

\newcommand{\Sn}{\ensuremath{\mathcal{S}_n}}

\newcommand{\M}{\ensuremath{\mathcal{M}}}
\newcommand{\G}{\ensuremath{\mathcal{G}}}
\newcommand{\T}{\ensuremath{\mathcal{T}}}
\newcommand{\Ta}{\ensuremath{\T^a}}
\newcommand{\Xt}[1]{\ensuremath{(#1_t)_{t\in[0,\infty)}}}
\newcommand{\Mt}{\Xt{\M}}
\newcommand{\Inv}{\ensuremath{\mathrm{Inv}}}
\newcommand{\I}{\ensuremath{\mathcal{I}}}
\newcommand{\J}{\ensuremath{J}}
\newcommand{\transp}[2]{\ensuremath{(#1~#2)}}

\date{May 10, 2022}

\title{Continuous-time Mallows processes}
\author{Beno\^it Corsini}
\date{}

\begin{document}

\begin{abstract}
    In this article, we introduce \textit{Mallows processes}, defined to be continuous-time c\`adl\`ag processes with Mallows distributed marginals. We show that such processes exist and that they can be restricted to have certain natural properties. In particular, we prove that there exists \textit{regular} Mallows processes, defined to have their inversions numbers $\Inv_j(\sigma)=|\{i\in[j-1]:\sigma(i)>\sigma(j)\}|$ be independent increasing stochastic processes with jumps of size $1$. We further show that there exists a unique Markov process which is a regular Mallows process. Finally, we study properties of regular Mallows processes and show various results on the structure of these objects. Among others, we prove that the graph structure related to regular Mallows processes looks like an \textit{expanded hypercube} where we stacked $k$ hypercubes on the dimension $k\in[n]$; we also prove that the first jumping times of regular Mallows processes converge to a Poisson point process.
\end{abstract}

\maketitle

\tableofcontents

\section{Introduction}\label{sec:introduction}

Write $\mathbb{N}=\{1,2,\ldots\}$ for the set of positive integers and let $\Sn$ be the set of permutations on $[n]=\{1,2,\ldots,n\}$. We typically represent a permutation $\sigma\in\Sn$ as a sequence of integers $(\sigma(1),\ldots,\sigma(n))$, so for example, the permutation $(3,1,4,2)$ sends $1$ to $3$, $2$ to $1$, $3$ to $4$ and $4$ to $2$. It is also useful to have more succinct notations for transpositions, so for $i\neq j$, we write $\transp{i}{j}$ for the transposition sending $i$ to $j$ and $j$ to $i$. Given two permutations $\sigma,\sigma'\in\Sn$, write $\sigma\cdot\sigma'$ for the composed permutation defined by $\sigma\cdot\sigma'(i)=\sigma(\sigma'(i))$.

For $q\in[0,\infty)$, the Mallows distribution~\cite{mallows1957non} with parameters $n$ and $q$ is the probability measure $\pi_{n,q}$ on $\Sn$ defined by
\begin{align*}
    \pi_{n,q}(\sigma)=\frac{q^{\Inv(\sigma)}}{Z_{n,q}}\,,
\end{align*}
where $\Inv(\sigma)=|\{i<j:\sigma(i)>\sigma(j)\}|$ is the number of inversions of $\sigma$, and $Z_{n,q}=\prod_{k=1}^n(\sum_{\ell=0}^{k-1}q^\ell)$ is a normalizing constant. Furthermore, write $\Inv_j(\sigma)=|\{i\in[j-1]:\sigma(i)>\sigma(j)\}|$ for the number of inversions of $\sigma$ created by $\sigma(j)$. Note that $\Inv(\sigma)=\sum_{j=1}^n\Inv_j(\sigma)$ and that $0\leq\Inv_j(\sigma)\leq j-1$.

We define a \textit{continuous-time Mallows process} (or simply \textit{Mallows process}) to be an $\Sn$-valued c\`adl\`ag stochastic process $\Xt{\M^n}$ such that, for all $t\in[0,\infty)$, $\M^n_t$ is $\pi_{n,t}$-distributed. For simplification, and since it is always clear from context, we drop the superscript $n$ and simply write $\Mt$.

Given a Mallows process $\Mt$, say that it is \textit{monotone} if, for all $t<t'$, we have
\begin{align*}
    \Inv(\M_t)\leq\Inv(\M_{t'})\,,
\end{align*}
and say that it is \textit{strongly monotone} if furthermore
\begin{align*}
    \Inv_j(\M_t)\leq\Inv_j(\M_{t'})
\end{align*}
for all $j\in[n]$. Moreover, say that $\Mt$ is \textit{smooth} if it is strongly monotone and for all $t\in[0,\infty)$, we have
\begin{align*}
    \Inv(\M_t)\leq\Inv(\M_{t-})+1\,.
\end{align*}
Finally say that it has \textit{independent inversions} if the processes $(\Inv_j(\M_t))_{t\in[0,\infty)}$ are independent over $j\in[n]$.

In this work we are mainly interested in defining and studying \textit{regular} Mallows processes, defined to be smooth Mallows processes with independent inversions. Our first result states that regular Mallows processes exist and that there is a unique such process with the Markov property.

\begin{thm}\label{thm:existenceMt}
    Fix $n\in\mathbb{N}$. There exists a unique c\`adl\`ag Markov process $\Mt$ on $\Sn$ which is a regular Mallows process.
\end{thm}

The proof of Theorem~\ref{thm:existenceMt} can be found in Section~\ref{subsec:twoMallows}; it boils down to showing that $\Mt$ is characterized by the processes $(\Inv_j(\M_t))_{[t,\infty)}$ for $j\in[n]$, and that such processes are time-inhomogeneous birth processes with a well-defined generator. For the rest of the paper, we refer to this unique process as the \textit{birth Mallows process}, since it will be defined using its birth processes.

Given a Mallows process $\M=\Mt$, let $\G_\M=(\Sn,E(\G_\M))$ be the undirected transition graph of $\M$ whose edge set is defined by
\begin{align*}
    E(\G_\M):=\Big\{(\sigma,\sigma'):\,\mathbb{P}\Big(\exists t\in(0,\infty),\M_{t-}=\sigma,\M_t=\sigma'\Big)>0\Big\}\,.
\end{align*}
In other words, $(\sigma,\sigma')\in E(\G_\M)$ if and only if it is possible for $\M$ to directly transition from $\sigma$ to $\sigma'$. For any graph $\G=(\Sn,E(\G))$, consider its \textit{generator} defined by
\begin{align*}
    \langle\G\rangle:=\big\{\sigma^{-1}\cdot\sigma':(\sigma,\sigma')\in E(\G)\big\}\,,
\end{align*}
so $\tau\in\langle\G\rangle$ if and only if $\G$ contains at least one edge of the form $(\sigma,\sigma\cdot\tau)$. Write $\T=\{\transp{i}{j}:i\neq j\}$ for the set of transpositions and $\Ta=\{\transp{i}{i+1}:i\in[n-1]\}\subset\T$ for the set of adjacent transpositions. Our next theorem controls the structure of $\G_\M$, and in particular the permutations that belong to the generator $\langle\G_\M\rangle$ when $\M$ is a smooth Mallows process. Let $\mathcal{H}_n=(\Sn,E)$ be the \textit{expended hypercube} of size $n$ defined by
\begin{align*}
    E=\left\{(\sigma,\sigma'):\sum_{j=1}^n\big|\Inv_j(\sigma)-\Inv_j(\sigma')\big|=1\right\}\,.
\end{align*}
The reason for naming this graph the expended hypercube comes from its structure. Indeed, this graph is isomorphic to $\G=(V,E)$ defined by
\begin{align*}
    V=\Big\{(x_1,\ldots,x_n):\forall i\in[n],0\leq x_i<i\Big\}
\end{align*}
and
\begin{align*}
    E=\left\{(x,x')\in V^2:\sum_{j=1}^n|x_j-x'_j|=1\right\}\,;
\end{align*}
the latter graph can be seen as a hypercube expanded $k$ times in dimension $k$. A representation of $\mathcal{H}_n$ for $n\in\{2,3,4\}$ can be found in Figure~\ref{fig:expandedHypercube}. The fact that $\mathcal{H}_n$ and $\G$ are isomorphic follows by considering the image of $\mathcal{H}_n$ under the map $\Phi$, defined in Section~\ref{sec:constructingMallows}.

\begin{figure}[htb]\label{fig:expandedHypercube}
    \centering
    \begin{tikzpicture}[scale=1.5]
        \begin{scope}
            \draw[line width=0.05cm, theme] (0,0) -- (1,0);
            \node[draw, circle, line width=0.05cm, theme, fill=lighttheme!80!white] at (0,0){};
            \node[draw, circle, line width=0.05cm, theme, fill=lighttheme!80!white] at (1,0){};
            \node[theme,scale=1.2] at (0.5,-0.5){$\mathcal{H}_2$};
            \node[scale=0.5, theme!50!black, anchor=south east, inner sep=0.3cm] at (0,0){$(1,2)$};
            \node[scale=0.5, theme!50!black, anchor=north west, inner sep=0.3cm] at (1,0){$(2,1)$};
        \end{scope}
        \begin{scope}[xshift=2.5cm]
            \draw[line width=0, lightblock, fill=lightblock, opacity=0.8] (0,0) -- (1,0) -- (1,2) -- (0,2) -- cycle;

            \draw[line width=0.05cm, theme] (0,0) -- (1,0);
            \draw[line width=0.05cm, theme] (0,1) -- (1,1);
            \draw[line width=0.05cm, theme] (0,2) -- (1,2);
            \draw[line width=0.05cm, theme] (0,0) -- (0,2);
            \draw[line width=0.05cm, theme] (1,0) -- (1,2);
            \node[draw, circle, line width=0.05cm, theme, fill=lighttheme!80!white] at (0,0){};
            \node[draw, circle, line width=0.05cm, theme, fill=lighttheme!80!white] at (1,0){};
            \node[draw, circle, line width=0.05cm, theme, fill=lighttheme!80!white] at (0,1){};
            \node[draw, circle, line width=0.05cm, theme, fill=lighttheme!80!white] at (1,1){};
            \node[draw, circle, line width=0.05cm, theme, fill=lighttheme!80!white] at (0,2){};
            \node[draw, circle, line width=0.05cm, theme, fill=lighttheme!80!white] at (1,2){};
            \node[theme,scale=1.2] at (0.5,-0.5){$\mathcal{H}_3$};
            \node[scale=0.5, theme!50!black, anchor=south east, inner sep=0.3cm] at (0,0){$(1,2,3)$};
            \node[scale=0.5, theme!50!black, anchor=south east, inner sep=0.3cm] at (0,1){$(1,3,2)$};
            \node[scale=0.5, theme!50!black, anchor=south east, inner sep=0.3cm] at (0,2){$(2,3,1)$};
            \node[scale=0.5, theme!50!black, anchor=north west, inner sep=0.3cm] at (1,0){$(2,1,3)$};
            \node[scale=0.5, theme!50!black, anchor=north west, inner sep=0.3cm] at (1,1){$(3,1,2)$};
            \node[scale=0.5, theme!50!black, anchor=north west, inner sep=0.3cm] at (1,2){$(3,2,1)$};
        \end{scope}
        \begin{scope}[xshift=5.5cm]
            \node[scale=0.5, theme!50!black, anchor=south east, inner sep=0.3cm] at (0,0){$(1,2,3,4)$};
            \node[scale=0.5, theme!50!black, anchor=south east, inner sep=0.3cm] at (0,1){$(1,3,2,4)$};
            \node[scale=0.5, theme!50!black, anchor=south east, inner sep=0.3cm] at (0,2){$(2,3,1,4)$};
            \node[scale=0.5, theme!50!black, anchor=south east, inner sep=0.3cm] at (0.65,0.3){$(1,2,4,3)$};
            \node[scale=0.5, theme!50!black, anchor=south east, inner sep=0.3cm] at (0.65,1.3){$(1,4,2,3)$};
            \node[scale=0.5, theme!50!black, anchor=south east, inner sep=0.3cm] at (0.65,2.3){$(2,4,1,3)$};
            \node[scale=0.5, theme!50!black, anchor=south east, inner sep=0.3cm] at (1.3,0.6){$(1,3,4,2)$};
            \node[scale=0.5, theme!50!black, anchor=south east, inner sep=0.3cm] at (1.3,1.6){$(1,4,3,2)$};
            \node[scale=0.5, theme!50!black, anchor=south east, inner sep=0.3cm] at (1.3,2.6){$(3,4,1,2)$};
            \node[scale=0.5, theme!50!black, anchor=south east, inner sep=0.3cm] at (1.95,0.9){$(2,3,4,1)$};
            \node[scale=0.5, theme!50!black, anchor=south east, inner sep=0.3cm] at (1.95,1.9){$(2,4,3,1)$};
            \node[scale=0.5, theme!50!black, anchor=south east, inner sep=0.3cm] at (1.95,2.9){$(3,4,2,1)$};

            \draw[line width=0, lightblock, fill=lightblock, opacity=0.4] (0,0) -- (1,0) -- (2.95,0.9) -- (2.95,2.9) -- (1.95,2.9) -- (0,2) -- cycle;

            \draw[line width=0.05cm, theme] (0.65,0.3) -- (1.65,0.3);
            \draw[line width=0.05cm, theme] (0.65,1.3) -- (1.65,1.3);
            \draw[line width=0.05cm, theme] (0.65,0.3) -- (0.65,2.3);
            \draw[line width=0.05cm, theme] (1.3,0.6) -- (2.3,0.6);
            \draw[line width=0.05cm, theme] (1.3,1.6) -- (2.3,1.6);
            \draw[line width=0.05cm, theme] (1.3,0.6) -- (1.3,2.6);
            \draw[line width=0.05cm, theme] (1.95,0.9) -- (2.95,0.9);
            \draw[line width=0.05cm, theme] (1.95,1.9) -- (2.95,1.9);
            \draw[line width=0.05cm, theme] (1.95,0.9) -- (1.95,2.9);
            \draw[line width=0.05cm, theme] (0,0) -- (1.95,0.9);
            \draw[line width=0.05cm, theme] (0,1) -- (1.95,1.9);

            \node[draw, circle, line width=0.05cm, theme, fill=lighttheme!80!white] at (0.65,0.3){};
            \node[draw, circle, line width=0.05cm, theme, fill=lighttheme!80!white] at (0.65,1.3){};
            \node[draw, circle, line width=0.05cm, theme, fill=lighttheme!80!white] at (1.3,0.6){};
            \node[draw, circle, line width=0.05cm, theme, fill=lighttheme!80!white] at (1.3,1.6){};
            \node[draw, circle, line width=0.05cm, theme, fill=lighttheme!80!white] at (1.95,0.9){};
            \node[draw, circle, line width=0.05cm, theme, fill=lighttheme!80!white] at (1.95,1.9){};

            \draw[line width=0, lightblock, fill=lightblock, opacity=0.6] (0,0) -- (1,0) -- (2.95,0.9) -- (2.95,2.9) -- (1.95,2.9) -- (0,2) -- cycle;

            \draw[line width=0.05cm, theme] (0,0) -- (1,0);
            \draw[line width=0.05cm, theme] (0,1) -- (1,1);
            \draw[line width=0.05cm, theme] (0,2) -- (1,2);
            \draw[line width=0.05cm, theme] (0,0) -- (0,2);
            \draw[line width=0.05cm, theme] (1,0) -- (1,2);

            \draw[line width=0.05cm, theme] (0.65,2.3) -- (1.65,2.3);
            \draw[line width=0.05cm, theme] (1.65,0.3) -- (1.65,2.3);
            \draw[line width=0.05cm, theme] (1.3,2.6) -- (2.3,2.6);
            \draw[line width=0.05cm, theme] (2.3,0.6) -- (2.3,2.6);
            \draw[line width=0.05cm, theme] (1.95,2.9) -- (2.95,2.9);
            \draw[line width=0.05cm, theme] (2.95,0.9) -- (2.95,2.9);

            \draw[line width=0.05cm, theme] (1,0) -- (2.95,0.9);
            \draw[line width=0.05cm, theme] (1,1) -- (2.95,1.9);
            \draw[line width=0.05cm, theme] (0,2) -- (1.95,2.9);
            \draw[line width=0.05cm, theme] (1,2) -- (2.95,2.9);

            \node[draw, circle, line width=0.05cm, theme, fill=lighttheme!80!white] at (0,0){};
            \node[draw, circle, line width=0.05cm, theme, fill=lighttheme!80!white] at (1,0){};
            \node[draw, circle, line width=0.05cm, theme, fill=lighttheme!80!white] at (0,1){};
            \node[draw, circle, line width=0.05cm, theme, fill=lighttheme!80!white] at (1,1){};
            \node[draw, circle, line width=0.05cm, theme, fill=lighttheme!80!white] at (0,2){};
            \node[draw, circle, line width=0.05cm, theme, fill=lighttheme!80!white] at (1,2){};

            \node[draw, circle, line width=0.05cm, theme, fill=lighttheme!80!white] at (1.65,0.3){};
            \node[draw, circle, line width=0.05cm, theme, fill=lighttheme!80!white] at (1.65,1.3){};
            \node[draw, circle, line width=0.05cm, theme, fill=lighttheme!80!white] at (0.65,2.3){};
            \node[draw, circle, line width=0.05cm, theme, fill=lighttheme!80!white] at (1.65,2.3){};

            \node[draw, circle, line width=0.05cm, theme, fill=lighttheme!80!white] at (2.3,0.6){};
            \node[draw, circle, line width=0.05cm, theme, fill=lighttheme!80!white] at (2.3,1.6){};
            \node[draw, circle, line width=0.05cm, theme, fill=lighttheme!80!white] at (1.3,2.6){};
            \node[draw, circle, line width=0.05cm, theme, fill=lighttheme!80!white] at (2.3,2.6){};

            \node[draw, circle, line width=0.05cm, theme, fill=lighttheme!80!white] at (2.95,0.9){};
            \node[draw, circle, line width=0.05cm, theme, fill=lighttheme!80!white] at (2.95,1.9){};
            \node[draw, circle, line width=0.05cm, theme, fill=lighttheme!80!white] at (1.95,2.9){};
            \node[draw, circle, line width=0.05cm, theme, fill=lighttheme!80!white] at (2.95,2.9){};

            \node[scale=0.5, theme!50!black, anchor=north west, inner sep=0.3cm] at (1,0){$(2,1,3,4)$};
            \node[scale=0.5, theme!50!black, anchor=north west, inner sep=0.3cm] at (1,1){$(3,1,2,4)$};
            \node[scale=0.5, theme!50!black, anchor=north west, inner sep=0.3cm] at (1,2){$(3,2,1,4)$};
            \node[scale=0.5, theme!50!black, anchor=north west, inner sep=0.3cm] at (1.65,0.3){$(2,1,4,3)$};
            \node[scale=0.5, theme!50!black, anchor=north west, inner sep=0.3cm] at (1.65,1.3){$(4,1,2,3)$};
            \node[scale=0.5, theme!50!black, anchor=north west, inner sep=0.3cm] at (1.65,2.3){$(4,2,1,3)$};
            \node[scale=0.5, theme!50!black, anchor=north west, inner sep=0.3cm] at (2.3,0.6){$(3,1,4,2)$};
            \node[scale=0.5, theme!50!black, anchor=north west, inner sep=0.3cm] at (2.3,1.6){$(4,1,3,2)$};
            \node[scale=0.5, theme!50!black, anchor=north west, inner sep=0.3cm] at (2.3,2.6){$(4,3,1,2)$};
            \node[scale=0.5, theme!50!black, anchor=north west, inner sep=0.3cm] at (2.95,0.9){$(3,2,4,1)$};
            \node[scale=0.5, theme!50!black, anchor=north west, inner sep=0.3cm] at (2.95,1.9){$(4,2,3,1)$};
            \node[scale=0.5, theme!50!black, anchor=north west, inner sep=0.3cm] at (2.95,2.9){$(4,3,2,1)$};
            \node[theme,scale=1.2] at (0.5,-0.5){$\mathcal{H}_4$};
        \end{scope}
    \end{tikzpicture}
    \caption{A visual representation of the expended hypercubes $\mathcal{H}_2$, $\mathcal{H}_3$, and $\mathcal{H}_4$. Using these figures, $\mathcal{H}_n$ can be seen as a stack of hypercubes where each dimension $k\in[n]$ has $k$ hypercubes on top of each other.}
\end{figure}
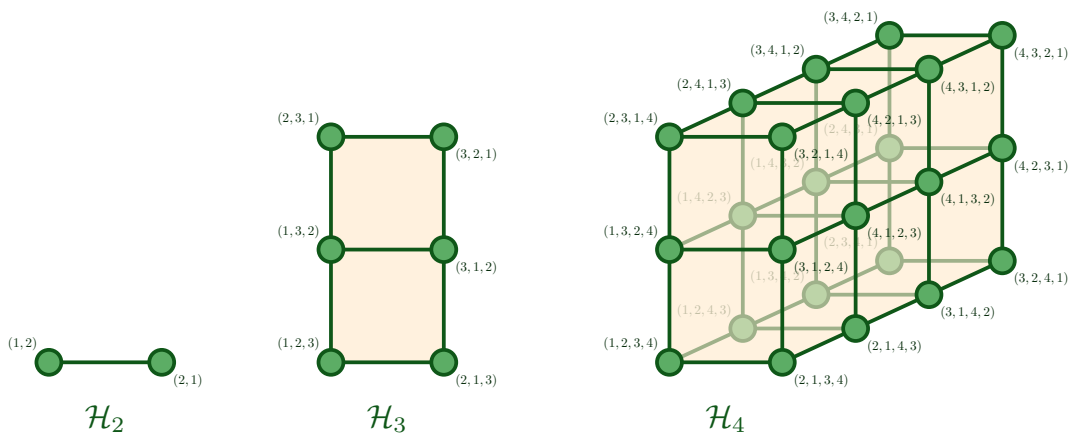

We now use $\T$, $\Ta$, and $\mathcal{H}_n$ to characterize the structure of $\G_\M$ and of $\langle\G_\M\rangle$ when $\M$ is a smooth Mallows process. The proof of next theorem appears in Section~\ref{subsec:MallowsGraph}.

\begin{thm}\label{thm:MallowsGraph}
    Let $\M=\Mt$ be a smooth Mallows process. Then $\G_\M$ is a subgraph of $\mathcal{H}_n$ and
    \begin{align*}
        \Ta\subset\langle\G_\M\rangle\subset\T\,.
    \end{align*}
    Moreover, if $\M$ is regular, then
    \begin{align*}
        \G_\M=\mathcal{H}_n\,,
    \end{align*}
    which implies that $\langle\G_\M\rangle=\langle\mathcal{H}_n\rangle=\T$.
\end{thm}

Consider now a monotone Mallows process $\Mt$. Since $(\Inv(\M_t))_{t\in[0,\infty)}$ is a jump process, we will be considering its \textit{jumping times} $T_1,\ldots,T_{\binom{n}{2}}$ defined as follows. For any $0\leq k\leq\binom{n}{2}$, let
\begin{align*}
    T_k:=\inf\Big\{t\in[0,\infty):\Inv(\M_t)\geq k\Big\}\,.
\end{align*}
Note that, without further assumptions, it is possible to have $T_{k-1}=T_k$ for some $k\geq1$. The next results characterizes the asymptotic behaviour of the early jumping times.

\begin{thm}\label{thm:PoissonTk}
    Let $\Mt$ be a monotone Mallows process and $T_1,\ldots,T_{\binom{n}{2}}$ be its jumping times. Then for any fixed $k\in\mathbb{N}$, as $n$ tends to infinity, we have
    \begin{align*}
        nT_k\overset{d}{\longrightarrow}\mathrm{Gamma}(k,1)\,.
    \end{align*}
    Moreover, if $\Mt$ is regular, then
    \begin{align*}
        (nT_1,\ldots,nT_k)\overset{d}{\longrightarrow}\mathrm{PP}_k(1)\,,
    \end{align*}
    where $\mathrm{PP}_k(\lambda)$ is the distribution of the first $k$ points of a homogeneous Poisson point process on $[0,\infty)$ with rate $\lambda$.
\end{thm}

The proof of Theorem~\ref{thm:PoissonTk} can be found in Section~\ref{subsec:convergencePoisson}, and boils down to first using the relation between number of jumps and inversion number in a monotone process, then showing that the first $k$ jumps of a regular Mallows process occur for distinct inversion numbers, making them independent of each other. It is worth noting that Theorem~\ref{thm:PoissonTk} does not assume that $\Mt$ is a Markov process.

The rest of this work is organized as follows. In Section~\ref{sec:constructingMallows} we give useful tools to construct regular Mallows processes and, in particular, define and prove the uniqueness of the birth Mallows process, hence proving Theorem~\ref{thm:existenceMt}. Section~\ref{sec:constructingMallows} also defines another interesting Mallows process which is not a Markov process and which we call the \textit{uniform Mallows process}; in particular, this process is defined as a function of only $n$ independent uniform random variables, implying that the amount of information necessary to generate it or known at a certain time can be computed. In Section~\ref{sec:propertiesMonotone} we establish various properties of Mallows processes and prove Theorems~\ref{thm:MallowsGraph} and \ref{thm:PoissonTk}. Section~\ref{sec:propertiesMonotone} also contains a section dedicated to properties of the amount of information related to the uniform Mallows process. Finally, we conclude this work with Section~\ref{sec:conclusion} containing a list of open problems and conjectures related to these newly defined Mallows processes, and in particular further properties we expect the birth Mallows process to hold.

\subsection{Related work}

The Mallows distribution of permutations was first introduced by C.L.~Mallows~\cite{mallows1957non} in the context of ranking theory. The study of its structural properties has gained a lot of interests in the last years and now contains literature related to various topics. Among the properties of these models that have attracted the most attention, it is worth mentioning the length of longest increasing subsequences~\cite{basu2017limit,bhatnagar2015lengths,jin2019limit,mueller2013length} as well as the cycle and subsequence structure~\cite{crane2018probability,gladkich2018cycle,he2020central,mukherjee2016fixed,pinsky2019permutations}, substantially studied over the last decade. Other works studying exchangeability~\cite{gnedin2012two,gnedin2010q}, random matchings~\cite{angel2018mallows}, binary search trees~\cite{addario2021height}, and colourings~\cite{holroyd2020mallows} have led to interesting insights on properties of these random permutations. Finally, Mallows permutations have found themselves as a useful model for various applications, such as convergence to stationarity~\cite{benjamini2005mixing,diaconis2004analysis}, statistical physics~\cite{starr2009thermodynamic,starr2018phase}, or even statistical learning~\cite{tang2019mallows} and genomics~\cite{fang2021arcsine}.

\section{Constructing regular Mallows process}\label{sec:constructingMallows}

In this section, we provide a natural framework for constructing regular Mallows processes and prove Theorem~\ref{thm:existenceMt}.

\subsection{An important bijection}\label{subsec:bijection}

Consider first the set $\mathcal{E}_n$ defined by
\begin{align*}
    \mathcal{E}_n=\Big\{(I_1,\ldots,I_n):\forall j\in[n], 0\leq I_j\leq j-1\Big\}\,,
\end{align*}
and let
\begin{align*}
    \Phi:\mathcal{E}_n&\longrightarrow\Sn\\
    I&\longmapsto\sigma\,,
\end{align*}
be the unique bijection between $\mathcal{E}_n$ and $\Sn$ such that $I_j=\Inv_j(\sigma)$. A simple way to generate $\Phi$ is by defining $[n]\setminus\{\sigma(n),\ldots,\sigma(k+1)\}=\{x_1<\ldots<x_k\}$ and setting $\sigma(k)=x_{k-I_k}$ (see~\cite{weselcouch2016patterns} for further details). This bijection shows to be an efficient way to easily generate Mallows permutations and Mallows processes, as shown in the next two propositions.

\begin{prop}\label{prop:ItoMallows}
    Fix $n\in\mathbb{N}$ and $q\in[0,\infty)$. Let $I=(I_1,\ldots,I_n)\in\mathcal{E}_n$ be a random sequence such that the different entries of $I$ are independent and, for all $j\in[n]$, for all $0\leq k\leq j-1$, we have
    \begin{align*}
        \mathbb{P}(I_j=k)=\frac{q^k}{\sum_{\ell=0}^{j-1}q^\ell}\,.
    \end{align*}
    Then $\Phi(I)$ is $\pi_{n,q}$-distributed.
\end{prop}

\begin{proof}
    Using the definition of $\Phi$, for any $\sigma\in\Sn$, we have
    \begin{align*}
        \mathbb{P}\Big(\Phi(I)=\sigma\Big)&=\mathbb{P}\Big(\forall j\in[n]:I_j=\Phi^{-1}(\sigma)_j\Big)=\prod_{j=1}^n\frac{q^{\Inv_j(\sigma)}}{\sum_{\ell=0}^{j-1}q^\ell}\,,
    \end{align*}
    where the second equality uses that $\Phi^{-1}(\sigma)=(\Inv_1(\sigma),\ldots,\Inv_n(\sigma))$ along with the fact that the entries of $I$ are independent. The proposition now follows from the definition of $\pi_{n,q}$.
\end{proof}

With this proposition, it becomes easy to generate regular Mallows processes, as stated in the following proposition.

\begin{prop}\label{prop:generateMallows}
    Fix $n\in\mathbb{N}$. Let $(I(t))_{t\in[0,\infty)}=(I_1(t),\ldots,I_n(t))_{t\in[0,\infty)}$ be a c\`adl\`ag process taking values in $\mathcal{E}_n$ such that the coordinate processes $\{(I_j(t))_{t\in[0,\infty)},j\in[n]\}$ are independent and, for all $j\in[n]$ and $0\leq k\leq j-1$, we have
    \begin{align*}
        \mathbb{P}\big(I_j(t)=k\big)=\frac{t^k}{\sum_{\ell=0}^{j-1}t^\ell}\,.
    \end{align*}
    Then $(\Phi(I(t)))_{t\in[0,\infty)}$ is a Mallows process with independent inversions. Moreover, if the processes $(I_j(t))_{t\in[0,\infty)}$ are increasing, then $(\Phi(I(t)))_{t\in[0,\infty)}$ is a strongly monotone Mallows process. Finally, if additionally the increments of the processes $(I_j(t))_{t\in[0,\infty)}$ are of size $1$, then $(\Phi(I(t)))_{t\in[0,\infty)}$ is a regular Mallows process.
\end{prop}

\begin{proof}
    Write $\Mt=(\Phi(I(t)))_{t\in[0,\infty)}$. The fact that $\Mt$ is a Mallows process simply follows from applying Proposition~\ref{prop:ItoMallows}. Similarly, if the processes $(I_j(t))_{t\in[0,\infty)}$ are increasing, then, by the definition of $\Phi$, $\Mt$ is strongly monotone.
    
    Assume now that the processes $(I_j(t))_{t\in[0,\infty)}$ are increasing with increments of size $1$. To prove that $\Mt$ is regular, it suffices to show that $(\Inv(\M_t))_{t\in[0,\infty)}$ has increments of size $1$.
    
    Let $T^j_k$ by the $k$-th jumping time of $I_j$, that is
    \begin{align*}
        T^j_k=\inf\big\{t:[0,\infty):I_j(t)\geq k\big\}\,.
    \end{align*}
    Using our assumption on $(I_j(t))_{t\in[0,\infty)}$, we know that $T^j_1<T^j_2<\ldots<T^j_{j-1}$. Moreover, the definition of $T^j_k$ implies that
    \begin{align*}
        \mathbb{P}\big(T^j_k\leq t\big)=\mathbb{P}\big(I_j(t)\geq k\big)=\frac{1}{\sum_{\ell=0}^{j-1}t^\ell}\sum_{\ell=k}^{j-1}t^\ell\,,
    \end{align*}
    so $T^j_k$ is a continuous random variable. Since the coordinate processes $((I_j(t))_{t\in[0,\infty)},j\in[n])$ are independent, it follows that for $j\neq j'$ almost surely $T^j_k\neq T^{j'}_{k'}$, which implies that all the values $\{T^j_k:j\in[n],k\in[j-1]\}$ are almost surely distinct. Using that
    \begin{align*}
        \Inv(\M_t)=\big|\big\{1\leq k<j\leq n:T^j_k\leq t\big\}\big|\,,
    \end{align*}
    it follows that $(\Inv(\M_t))_{t\in[0,\infty)}$ only has increments of size $1$.
\end{proof}

Proposition~\ref{prop:generateMallows} is a useful tool for building regular Mallows processes; we develop this in the next section.

\subsection{Two regular Mallows processes}\label{subsec:twoMallows}

In this section, we construct two different regular Mallows processes, both based on defining an appropriate process $(I(t))_{t\in[0,\infty)}$ and applying Proposition~\ref{prop:generateMallows}. For the rest of this section, fix $n\in\mathbb{N}$.

\subsubsection{Birth Mallows process}\label{subsec:birthMallows}

The construction of the (unique) birth Mallows process, is based on the time-inhomogeneous birth processes defined in the following lemma.

\begin{lemma}\label{lem:birthModel}
    Fix $j\in\mathbb{N}$. For any $0\leq k\leq j-1$ and $t\in[0,\infty)$, let
    \begin{align*}
        p_j(t,k):=\frac{1}{\sum_{\ell=0}^{j-1}t^\ell}\left[(k+1)\sum_{\ell=0}^{j-2}(\ell+1)t^\ell-j\sum_{\ell=j-k-2}^{j-2}(\ell-j+k+2)t^\ell\right]\,.
    \end{align*}
    Define now the time-inhomogenous birth process $(B_j(t))_{t\in[0,\infty)}$ by $B_j(0)=0$ and infinitesimal rates given by
    \begin{align*}
        \mathbb{P}\Big(B_j(t+h)=k+\ell~\Big|~ B_j(t)=k\Big)=\left\{\begin{array}{ll}
        1-hp_j(t,k)+o(h) & \textrm{if $\ell=0$} \\
        hp_j(t,k)+o(h) & \textrm{if $\ell=1$} \\
        o(h) & \textrm{otherwise}\,,
        \end{array}\right.
    \end{align*}
    where $o(h)/h\rightarrow0$ as $h\rightarrow0$.
    Then for all $0\leq k\leq j-1$ and all $t\in[0,\infty)$, we have
    \begin{align*}
        \mathbb{P}\Big(B_j(t)=k\Big)=\frac{t^k}{\sum_{\ell=0}^{j-1}t^\ell}\,.
    \end{align*}
\end{lemma}

\begin{proof}
    First, note that
    \begin{align*}
        p_j(t,j-1)=\frac{1}{\sum_{\ell=0}^{j-1}t^\ell}\left[j\sum_{\ell=0}^{j-2}(\ell+1)t^\ell-j\sum_{\ell=-1}^{j-2}(\ell+1)t^\ell\right]=0\,,
    \end{align*}
    which supports the fact that $B_j(t)\leq j-1$. Now, before proving the result, we verify that $p_j$ is non-negative, a necessary assumption for this birth process to be well-defined. For $0\leq k\leq j-2$ and $t\in[0,\infty)$, we have
    \begin{align*}
        &(k+1)\sum_{\ell=0}^{j-2}(\ell+1)t^\ell-j\sum_{\ell=j-k-2}^{j-2}(\ell-j+k+2)t^\ell\\
        &\hspace{0.5cm}=(k+1)\sum_{\ell=0}^{j-k-3}(\ell+1)t^\ell+\sum_{\ell=j-k-2}^{j-2}\Big[(k+1)(\ell+1)-j(\ell-j+k+2)\Big]t^\ell\,.
    \end{align*}
    The first sum is non-negative, and by re-organizing the terms inside the second sum, we have
    \begin{align*}
        (k+1)(\ell+1)-j(\ell-j+k+2)&=k+1-j(k+2-j)-\ell(j-k-1)\,.
    \end{align*}
    Since $j-k-1\geq1>0$, the previous equation is decreasing in $\ell$. Using that $\ell\leq j-2$, it follows that
    \begin{align*}
        (k+1)(\ell+1)-j(\ell-j+k+2)&\geq(k+1)(j-2+1)-j(j-2-j+k+2)\\
        &=j-k-1>0\,,
    \end{align*}
    proving that $p_j(t,k)\geq0$.
    
    We now want to prove that $B_j(t)$ has the desired distribution. To do so, we prove by induction on $k$ that
    \begin{align*}
        \mathbb{P}\Big(B_j(t)=k\Big)=\frac{t^k}{\sum_{\ell=0}^{j-1}t^\ell}\,.
    \end{align*}
    For $k=0$, using that
    \begin{align*}
        \mathbb{P}\Big(B_j(t+h)=0~\Big|~B_j(t)=0\Big)=1-hp_j(t,0)+o(h)\,,
    \end{align*}
    we have that
    \begin{align*}
        \frac{\partial}{\partial t}\Big(\mathbb{P}\Big(B_j(t)=0\Big)\Big)=-p_j(t,0)\mathbb{P}\Big(B_j(t)=0\Big)\,.
    \end{align*}
    We solve this equation using that $\mathbb{P}(B_j(0)=0)=1$ and that
    \begin{align*}
        p_j(t,0)=\frac{1}{\sum_{\ell=0}^{j-1}t^\ell}\sum_{\ell=0}^{j-2}(\ell+1)t^\ell=\frac{\partial}{\partial t}\left(\log\left(\sum_{\ell=0}^{j-1}t^\ell\right)\right)\,,
    \end{align*}
    to obtain
    \begin{align*}
        \mathbb{P}\Big(B_j(t)=0\Big)=\frac{1}{\sum_{\ell=0}^{j-1}t^\ell}\,.
    \end{align*}
    
    Assume now that for some $0\leq k\leq j-2$, we have
    \begin{align*}
        \mathbb{P}\Big(B_j(t)=k\Big)=\frac{t^k}{\sum_{\ell=0}^{j-1}t^\ell}\,.
    \end{align*}
    Then, the forward equation tells us that
    \begin{align*}
        \mathbb{P}\Big(B_j(t+h)=k+1\Big)&=\sum_{\ell=0}^{k+1}\mathbb{P}\Big(B_j(t+h)=k+1~\Big|~B_j(t)=\ell\Big)\mathbb{P}\Big(B_j(t)=\ell\Big)\\
        &=\Big[1-hp_j(t,k+1)\Big]\mathbb{P}\Big(B_j(t)=k+1\Big)+hp_j(t,k)\frac{t^k}{\sum_{\ell=0}^{j-1}t^\ell}+o(h)\,,
    \end{align*}
    from which it follows that
    \begin{align}
        \frac{\partial}{\partial t}\Big(\mathbb{P}\Big(B_j(t)=k+1\Big)\Big)=-p_j(t,k+1)\mathbb{P}\Big(B_j(t)=k+1\Big)+p_j(t,k)\frac{t^k}{\sum_{\ell=0}^{j-1}t^\ell}\,.\label{eq:Bj}
    \end{align}

    Consider the target function for $\mathbb{P}(B_j(t)=k+1)$:
    \begin{align}
        f(t)=\frac{t^{k+1}}{\sum_{\ell=0}^{j-1}t^\ell}\,.\label{eq:f}
    \end{align}
    Note that
    \begin{align}
        -p_j(t,k+1)f(t)+p_j(t,k)\frac{t^k}{\sum_{\ell=0}^{j-1}t^\ell}&=\frac{t^k}{\sum_{\ell=0}^{j-1}t^\ell}\Big[p_j(t,k)-tp_j(t,k+1)\Big]\,,\label{eq:p and f}
    \end{align}
    from which, using the definition of $p_j$, we obtain
    \begin{align*}
        p_j(t,k)-tp_j(t,k+1)&=\frac{1}{\sum_{\ell=0}^{j-1}t^\ell}\left[(k+1)\sum_{\ell=0}^{j-2}(\ell+1)t^\ell-j\sum_{\ell=j-k-2}^{j-2}(\ell-j+k+2)t^\ell\right]\\
        &\hspace{0.5cm}-\frac{1}{\sum_{\ell=0}^{j-1}t^\ell}\left[(k+2)\sum_{\ell=0}^{j-2}(\ell+1)t^{\ell+1}-j\sum_{\ell=j-k-3}^{j-2}(\ell-j+k+3)t^{\ell+1}\right]\,.
    \end{align*}
    Using that
    \begin{align*}
        (k+1)\sum_{\ell=0}^{j-2}(\ell+1)t^\ell-(k+2)\sum_{\ell=0}^{j-2}(\ell+1)t^{\ell+1}&=-(k+2)(j-1)t^{j-1}+\sum_{\ell=0}^{j-2}(k-\ell+1)t^\ell\,,
    \end{align*}
    and that
    \begin{align*}
        -j\sum_{\ell=j-k-2}^{j-2}(\ell-j+k+2)t^\ell+j\sum_{\ell=j-k-3}^{j-2}(\ell-j+k+3)t^{\ell+1}=j(k+1)t^{j-1}\,,
    \end{align*}
    we obtain
    \begin{align*}
        p_j(t,k)-tp_j(t,k+1)&=\frac{1}{\sum_{\ell=0}^{j-1}t^\ell}\sum_{\ell=0}^{j-1}(k-\ell+1)t^\ell\,.
    \end{align*}

    Plug the previous formula into \eqref{eq:p and f} and use the definition of $f$ from \eqref{eq:f} to obtain that
    \begin{align*}
        -p_j(t,k+1)f(t)+p_j(t,k)\frac{t^k}{\sum_{\ell=0}^{j-1}t^\ell}&=\frac{t^k}{\left(\sum_{\ell=0}^{j-1}t^\ell\right)^2}\left[\sum_{\ell=0}^{j-1}(k+1-\ell)t^\ell\right]=f'(t)\,.
    \end{align*}
    Note now that this last equation is the same as \eqref{eq:Bj} where the probability $\mathbb{P}(B_j(t)=k+1)$ is replaced by $f(t)$. Using that $f(0)=\mathbb{P}(B_j(0)=k+1)=0$, it follows that $\mathbb{P}(B_j(t)=k+1)=f(t)$ which completes the inductive step. This concludes the proof of the lemma.
\end{proof}

Using Lemma~\ref{lem:birthModel}, we can now define the birth Mallows process.

\begin{prop}\label{prop:birthModel}
    Let $(B(t))_{t\in[0,\infty)}=(B_1(t),\ldots,B_n(t))_{t\in[0,\infty)}$, where $(B_j(t))_{t\in[0,\infty)}$ is defined as in Lemma~\ref{lem:birthModel} and let the birth Mallows process be defined by
    \begin{align*}
        \Xt{\M^B}:=(\Phi(B(t)))_{t\in[0,\infty)}\,.
    \end{align*}
    Then $\Xt{\M^B}$ is a regular Mallows process and a c\`adl\`ag Markov process.
\end{prop}

\begin{proof}
    By the definition of $B_1,\ldots,B_n$ from Lemma~\ref{lem:birthModel}, we know that $(B(t))_{t\in[,\infty)}$ is a c\`adl\`ag process of $\mathcal{E}_n$ satisfying all conditions of Proposition~\ref{prop:generateMallows}. This implies that $\Xt{\M^B}$ is a regular Mallows process. To see that it is also a Markov process, recall that $\Phi$ is a bijection, meaning that
    \begin{align*}
        \big\{\M^B_t=\sigma\big\}=\big\{B(t)=\Phi^{-1}(\sigma)\big\}=\big\{(B_1(t),\ldots,B_n(t))=\Phi^{-1}(\sigma)\big\}\,.
    \end{align*}
    Now, since all the processes $(B_1(t))_{t\in[0,\infty)},\ldots,(B_n(t))_{t\in[0,\infty)}$ are time-inhomogeneous birth processes, in particular they are Markov processes. This implies that, for any $t_1<\ldots<t_k<t$ and $\sigma_1,\ldots,\sigma_k,\sigma\in\Sn$, we have
    \begin{align*}
        &\hspace{-0.5cm}\mathbb{P}\Big(\M^B_t=\sigma~\Big|~\M^B_{t_1}=\sigma_1,\ldots,\M^B_{t_k}=\sigma_k\Big)\\
        &=\mathbb{P}\Big(B(t)=\Phi^{-1}(\sigma)~\Big|~B(t_1)=\Phi^{-1}(\sigma_1),\ldots,B(t_k)=\Phi^{-1}(\sigma_k)\Big)\\
        &=\prod_{j=1}^n\mathbb{P}\Big(B_j(t)=\Phi^{-1}(\sigma)_j~\Big|~B_j(t_1)=\Phi^{-1}(\sigma_1)_j,\ldots,B_j(t_k)=\Phi^{-1}(\sigma_k)_j\Big)\,,
    \end{align*}
    where the second inequality follows from independence of the coordinate processes of $(B(t))_{t\in[0,\infty)}$. But then, using that these are Markov processes, we obtain
    \begin{align*}
        \mathbb{P}\Big(\M^B_t=\sigma~\Big|~\M^B_{t_1}=\sigma_1,\ldots,\M^B_{t_k}=\sigma_k\Big)&=\prod_{j=1}^n\mathbb{P}\Big(B_j(t)=\Phi^{-1}(\sigma)_j~\Big|~B_j(t_k)=\Phi^{-1}(\sigma_k)_j\Big)\\
        &=\mathbb{P}\Big(\M^B_t=\sigma~\Big|~\M^B_{t_k}=\sigma_k\Big)\,,
    \end{align*}
    which exactly corresponds to the fact that $\Xt{\M^B}$ is a Markov process.
\end{proof}

In order to conclude the proof of Theorem~\ref{thm:existenceMt}, it only remains to prove that the birth Mallows process is the unique regular Mallows process which is also a Markov process.

\begin{proof}[Proof of Theorem~\ref{thm:existenceMt}]
    Thanks to Proposition~\ref{prop:birthModel}, we know that there exists a c\`adl\`ag Markov process which is a regular Mallows process.
    
    Conversely, using $\Phi$ as defined in Section~\ref{sec:constructingMallows}, regular Mallows processes are fully characterized by the processes $(\Inv_1(\M_t))_{t\in[0,\infty)},(\Inv_2(\M_t))_{t\in[0,\infty)},\ldots,(\Inv_n(\M_t))_{t\in[0,\infty)}$, since regular Mallows processes have the independent inversion property, these processes are independent jump processes. Moreover, using that a regular Mallows process is by definition smooth, we know that these processes will only jump by increments of $1$. This means that each $(\Inv_j(\M_t))_{t\in[0,\infty)}$ for $j\in[n]$ is a time-inhomogeneous birth process, for which properties of time-homogenous birth process apply by adapting the formulas to incorporate the time parameter (see~\cite[Section~6.8]{grimmett2020probability} for the properties of time-homogeneous birth processes, which naturally extend to time-inhomogenous birth processes). In particular, for any $j\in[n]$, the time-inhomogenous birth process $(\Inv_j(\M_t))_{t\in[0,\infty)}$ is fully characterized by its infinitesimal generator
    \begin{align*}
        g_j(t,k)=\lim_{h\downarrow0}\frac{1-\mathbb{P}\big(\Inv_j(\M_{t+h})=k~\big|~\Inv_j(\M_t)=k\big)}{h}\,,
    \end{align*}
    just like in the case of time-homogeneous processes. Furthermore, these incremental generators $g_j$ have to satisfy the forward equation
    \begin{align*}
        \frac{\partial}{\partial t}\Big(\mathbb{P}\Big(\Inv_j(\M_t)=k+1\Big)\Big)=g_j(t,k)\mathbb{P}\Big(\Inv_j(\M_t)=k\Big)-g_j(t,k+1)\mathbb{P}\Big(\Inv_j(\M_t)=k+1\Big)\,.
    \end{align*}
    Since the previous equation is the same as \eqref{eq:Bj} with $p_j$ replaced by $g_j$, it follows that $g_j=p_j$ for all $j\in[n]$. This proves the uniqueness of such process.
\end{proof}

\subsubsection{Uniform Mallows process}\label{subsec:uniformMallows}

We now define a second Mallows process, referred to as the \textit{uniform Mallows process}, written $\Xt{\M^U}$, which is based on using a set of $n$ uniform random variables along with the following lemma.

\begin{lemma}\label{lem:uniformModel}
    Fix $j\in\mathbb{N}$ and $u\in(0,1)$, and define a function $f_j(\,\cdot\,;u):[0,\infty)\mapsto\{0,1,\ldots,j-1\}$ by
    \begin{align*}
        f_j(t;u)=\left\{\begin{array}{ll}
        \left\lfloor\frac{\log(1-u(1-t^j))}{\log t}\right\rfloor & \textrm{if $t\notin\{0,1\}$} \\
        0 & \textrm{if $t=0$} \\
        \lfloor ju\rfloor & \textrm{if $t=1$} \\
    \end{array}\right.\,.
    \end{align*}
    Then $f_j(\,\cdot\,;u)$ is increasing. Moreover, if $U$ is a $\mathrm{Uniform}([0,1])$ random variable, then for all $0\leq k\leq j-1$ and all $t\in[0,\infty)$, we have
    \begin{align*}
        \mathbb{P}\Big(f_j(t;U)=k\Big)=\frac{t^k}{\sum_{\ell=0}^{j-1}t^\ell}\,.
    \end{align*}
\end{lemma}

\begin{proof}
    First of all, consider $\phi(t)=\frac{\log(1-u(1-t^j))}{\log t}$. Then we have
    \begin{align*}
        \phi'(t)&=\frac{1}{(\log t)^2}\left[\frac{jut^{j-1}}{1-u(1-t^j)}\log t-\frac{1}{t}\log\Big(1-u(1-t^j)\Big)\right]\\
        &=\frac{1}{t(\log t)^2}\left[\frac{ut^j}{1-u(1-t^j)}\log(t^j)-\log\Big(1-u(1-t^j)\Big)\right]\,.
    \end{align*}
    Now, by the concavity of $\log$, we know that
    \begin{align*}
        \lambda\log\frac{1}{x}+(1-\lambda)\log\frac{1}{y}\geq\log\left(\frac{1}{\lambda x+(1-\lambda)y}\right)\,.
    \end{align*}
    Apply this formula with $\lambda=\frac{ut^j}{1-u(1-t^j)}$, $x=\frac{1}{t^j}$ and $y=1$ to obtain
    \begin{align*}
        \frac{ut^j}{1-u(1-t^j)}\log(t^j)\geq\log\left(\frac{1}{\frac{u}{1-u(1-t^j)}+\frac{1-u}{1-u(1-t^j)}}\right)=\log\Big(1-u(1-t^j)\Big)\,,
    \end{align*}
    which proves that $\phi'(t)\geq0$. This proves that $f_j(\,\cdot\,;u)$ is increasing on $(0,1)$ and on $(1,\infty)$. To conclude the proof of the first claim of the lemma, note that $f_j(\,\cdot\,;u)$ is also continuous at $0$ and $1$.
    
    For the second part of the lemma, first note that the statement holds for $t=0$ and $t=1$. Then, for $t>0$ with $t\neq1$, and for all $0\leq k\leq j-1$, we have
    \begin{align*}
        \mathbb{P}\Big(f_j(t;U)=k\Big)&=\mathbb{P}\left(k\leq \frac{\log(1-U(1-t^j))}{\log t}<k+1\right)\\
        &=\mathbb{P}\left(\frac{t^k-1}{t^j-1}\leq U<\frac{t^{k+1}-1}{t^j-1}\right)\\
        &=\frac{t^k(t-1)}{t^j-1}\,,
    \end{align*}
    which is the desired formula. This concludes the proof of the proposition.
\end{proof}

Using this lemma, we now define the uniform Mallows process.

\begin{deft}\label{def:uniformModel}
    Let $U_1,\ldots,U_n$ be a set of $n$ independent $\mathrm{Uniform}([0,1])$ random variables and let $I(t)=(f_1(t;U_1),\ldots,f_n(t;U_n))$, where $f_1,\ldots,f_n$ are as in Lemma~\ref{lem:uniformModel}. Then the uniform Mallows process is defined by $\Xt{\M^U}=(\Phi(I(t)))_{t\in[0,\infty)}$.
\end{deft}

Thanks to Lemma~\ref{lem:uniformModel} and Proposition~\ref{prop:generateMallows}, the process $\Xt{\M^U}$ is a regular Mallows process. However, this process is not a Markov process. Its properties will be further explored in Section~\ref{subsec:propertiesUniform}.

\section{Properties of regular Mallows processes}\label{sec:propertiesMonotone}

In this section, we prove Theorem~\ref{thm:MallowsGraph} and \ref{thm:PoissonTk}, and study further properties of the uniform Mallows process from Definition~\ref{def:uniformModel}.

\subsection{Transition graph}\label{subsec:MallowsGraph}

We start by proving our results regarding the graph structure of smooth and regular Mallows processes.

\begin{proof}[Proof of Theorem~\ref{thm:MallowsGraph}]
    The first part of the proof addresses the case where $\M$ is smooth; the second part proves the assertions of the theorem for regular $\M$. Recall from Section~\ref{subsec:bijection} that for any $I\in\mathcal{E}_n$ and for $\sigma=\Phi(I)$, the values $\sigma(n),\ldots,\sigma(1)$ can be recursively constructed by letting
    \begin{align}
        [n]\setminus\{\sigma(n),\ldots,\sigma(j+1)\}=\{x_1<\ldots<x_j\}\label{eq:recursive Phi}
    \end{align}
    and by setting $\sigma(j)=x_{j-I_j}$.
    
    For the first part of this proof, assume that $\M=\Mt$ is a smooth Mallows process. Note that the definition of the expanded hypercube directly implies that $\G_\M$ is a subgraph of $\mathcal{H}$, since every jump of $\Mt$ corresponds to an increment of exactly one of its inversion numbers by $1$, and by definition every such increment corresponds to an edge of $\mathcal{H}_n$.
    
    We now prove that $\langle\G_\M\rangle\subset\T$. Since $\G_\M$ is a subgraph of $\mathcal{H}_n$, it suffices to prove that the generator of $\mathcal{H}_n$ is a subset of $\T$. Using the definition of $\mathcal{H}_n$, the previous statement can be equivalently re-stated as follows. For any $I=(I_1,\ldots,I_n)\in\mathcal{E}_n$ and any $\ell\in[n]$ such that $I_\ell<\ell-1$, let $I'=(I'_1,\ldots,I'_n)\in\mathcal{E}_n$ be defined by $I'_j=I_j$ for $j\neq\ell$, and $I'_\ell=I_\ell+1$. Then $\Phi(I')\cdot(\Phi(I))^{-1}$ is a transposition. We now prove that this second and equivalent statement holds.
    
    Let $I$ and $I'$ be defined as above and let $\sigma=\Phi(I)$ and $\sigma'=\Phi(I')$. Let us try to understand the relation between $\sigma$ and $\sigma'$. Using the characterization of $\Phi$ we gave at the beginning of the proof, it is easy to verify that $\sigma(j)=\sigma'(j)$ for all $j>\ell$, since $I_j=I'_j$. Now, at step $\ell$, we have
    \begin{align*}
        [n]\setminus\{\sigma(\ell+1),\ldots,\sigma(n)\}=[n]\setminus\{\sigma'(\ell+1),\ldots,\sigma'(n)\}=\{x_1<\ldots<x_\ell\}\,.
    \end{align*}
    By definition, we then have $\sigma(\ell)=x_{\ell-I_\ell}$ and $\sigma'(\ell)=x_{\ell-I'_\ell}=x_{\ell-I_\ell-1}$. From this, it follows that
    \begin{align*}
        [n]\setminus\{\sigma(\ell),\ldots,\sigma(n)\}&=\Big([n]\setminus\{\sigma(\ell+1),\ldots,\sigma(n)\}\Big)\setminus\{x_{\ell-I_\ell}\}\\
        &=\{x_1<\ldots<x_{\ell-I_\ell-2}<x_{\ell-I_\ell-1}<x_{\ell-I_\ell+1}<\ldots<x_\ell\}
    \end{align*}
    and
    \begin{align*}
        [n]\setminus\{\sigma'(\ell),\ldots,\sigma'(n)\}&=\Big([n]\setminus\{\sigma(\ell+1),\ldots,\sigma(n)\}\Big)\setminus\{x_{\ell-I_\ell-1}\}\\
        &=\{x_1<\ldots<x_{\ell-I_\ell-2}<x_{\ell-I_\ell}<x_{\ell-I_\ell+1}<\ldots<x_\ell\}\,.
    \end{align*}
    Since $I_{\ell-1}=I'_{\ell-1}$ two cases follow. Either $I_{\ell-1}=I'_{\ell-1}=I_\ell$ and then $\sigma(\ell-1)=x_{\ell-I_\ell-1}$ and $\sigma'(\ell-1)=x_{\ell-I_\ell}$, meaning that $\sigma(\ell-1)=\sigma'(\ell)$ and $\sigma'(\ell-1)=\sigma(\ell)$; or $I_{\ell-1}=I'_{\ell-1}\neq I_\ell$, and then $\sigma(\ell-1)=\sigma'(\ell-1)$. In this second case, we have that
    \begin{align*}
        [n]\setminus\{\sigma(\ell-1),\ldots,\sigma(n)\}&=\{x_1<\ldots<x_{\ell-I_\ell-1}<x_{\ell-I_\ell+1}<\ldots<x_\ell\}\setminus\{\sigma(\ell-1)\}
    \end{align*}
    and that
    \begin{align*}
        [n]\setminus\{\sigma'(\ell-1),\ldots,\sigma'(n)\}&=\{x_1<\ldots<x_{\ell-I_\ell-2}<x_{\ell-I_\ell}<\ldots<x_\ell\}\setminus\{\sigma(\ell-1)\}
    \end{align*}
    By applying the same argument as before, depending on the value of $I_{\ell-2}=I'_{\ell-2}$, we obtain that either $\sigma(\ell-2)=x_{\ell-I_\ell-1}=\sigma'(\ell)$ and $\sigma'(\ell-2)=x_{\ell-I_\ell}=\sigma(\ell)$, or $\sigma(\ell-2)=\sigma'(\ell-2)\notin\{x_{\ell-I_\ell},x_{\ell-I_\ell-1}\}$. By repeating this argument, it follows that there exists $i\in[\ell-1]$ such that $\sigma(\ell-i)=\sigma'(\ell)$, $\sigma'(\ell-i)=\sigma(\ell)$, and for all $\ell-i<j<\ell$, $\sigma(j)=\sigma'(j)$. Note that we necessarily have that $i\leq\ell-1$ since, in the case where $i>\ell-2$, then $\sigma(j)=\sigma'(j)$ for all $\ell-i<j<\ell$. Now, using that $\ell-i\leq 1$, it follows that $\sigma(j)=\sigma'(j)$ for all $2\leq j<\ell$; but we also know that this equality holds for all $j>\ell$ and then, since $\sigma$ and $\sigma'$ are permutation and since $\sigma(\ell)\neq\sigma'(\ell)$, it follows that $\sigma(1)=\sigma'(\ell)$ and $\sigma(\ell)=\sigma'(1)$, hence $i=\ell-1$.
    Write now $k=\ell-i$ and note that
    \begin{align*}
        [n]\setminus\{\sigma(k),\ldots,\sigma(n)\}=[n]\setminus\{\sigma'(k),\ldots,\sigma'(n)\}\,,
    \end{align*}
    and that, for all $j<k$, we have $I_j=I'_j$. These two statement along with the description of $\Phi$ given in \eqref{eq:recursive Phi} imply that $\sigma(j)=\sigma'(j)$ for all $j<k$. This shows that, for all $j\notin\{k,\ell\}$, $\sigma(j)=\sigma'(j)$, and that $\sigma(\ell)=\sigma'(k)$ and $\sigma(k)=\sigma'(\ell)$. This is equivalent to saying that $\sigma'$ is obtained from $\sigma$ by transposing $k$ and $\ell$, or in other words, that $\sigma'=\sigma\cdot\transp{k}{\ell}$. This proves that $\langle\mathcal{H}_n\rangle\subset\T$, hence showing that $\langle\G_\M\rangle\subset\T$.
    
    We now prove that $\Ta\subset\langle\G_\M\rangle$. Note first that $\Ta=\{\sigma\in\Sn:\Inv(\sigma)=1\}$. Suppose that $\Ta\setminus\langle\G_\M\rangle\neq\emptyset$ and consider $\tau\in\Ta\setminus\langle\G_\M\rangle$. By the definition of $\langle\G_\M\rangle$ and $\G_\M$, we have
    \begin{align*}
        \mathbb{P}\big(\M_t=\tau\big)=\mathbb{P}\big(\M_t=\tau,\exists u>0:\M_{u-}=\mathrm{id},\M_u=\tau\big)\leq\mathbb{P}\big(\exists u>0:\M_{u-}=\mathrm{id},\M_u=\tau\big)=0\,.
    \end{align*}
    However, the first probability on the left is non-zero whenever $t>0$, creating a contradiction. Hence $\Ta\subset\langle\G_\M\rangle$
    
    We now move on to the second part of the proof and prove that if $\M$ is a regular Mallows process, then $\G_\M=\mathcal{H}_n$ and $\langle\G_\M\rangle=\langle\mathcal{H}_n\rangle=\T$. We start by proving that $\G_\M=\mathcal{H}_n$. Since we already proved that $\G_\M$ is a subgraph of $\mathcal{H}_n$, it suffices to show that $\mathcal{H}_n$ is a subgraph of $\G_\M$.
    
    Consider $(\sigma,\sigma')\in E(\mathcal{H}_n)$ and order $\sigma$ and $\sigma'$ so that $\Inv(\sigma')=\Inv(\sigma)+1$; write $I=(I_1,\ldots,I_n)=\Phi^{-1}(\sigma)$ and $I'=(I'_1,\ldots,I'_n)=\Phi^{-1}(I')$. Using the definition of $\mathcal{H}_n$, we know that there exists $\ell\in[n]$ such that $I'_j=I_j$ for all $j\neq \ell$ and $I'_{\ell}=I_{\ell}+1$. For $j\in[n]$, write $(T^j_k)_{k\in[j-1]}$ for the jumping times of $(\Inv_j(\M_t))_{t\in[0,\infty)}$, so
    \begin{align*}
        T^j_k=\inf\big\{t\in[0,\infty):\Inv_j(\M_t)\geq k\big\}\,.
    \end{align*}
    For convenience, set $T^j_0=0$ and $T^j_k=\infty$ for $k\geq j$. Note that the processes $((T^j_k)_{k\in[j-1]})_{j\in[n]}$ are mutually independent and all have continuous distributions. Consider now the following event
    \begin{align*}
        E=\Big\{\max\big\{T^j_{I_j}:j\in[n]\big\}<T^{\ell}_{I_{\ell+1}}<\min\big\{T^j_{I_j+1}:j\neq \ell\big\}\Big\}\,.
    \end{align*}
    To understand the role of $E$, let us first show that, if $\mathbb{P}(E)>0$, then $(\sigma,\sigma')\in E(\G_\M)$ (we will later prove that $\mathbb{P}(E)>0$). Assume that $E$ holds. Then we can find $t>0$ so that
    \begin{align*}
        \max\big\{T^j_{I_j}:j\in[n]\big\}<t<T^{\ell}_{I_{j_1+1}}<\min\big\{T^j_{I_j+1}:j\neq \ell\big\}\,.
    \end{align*}
    Now, at time $t$, we have that $\Inv_j(\M_t)=I_j$ for all $j\in[n]$ since $T^j_{I_j}<t<T^j_{I_j+1}$, implying that $\M_t=\Phi(I)=\sigma$. Then, at time $T^{\ell}_{I_{\ell}+1}$, the $\ell$-th inversion number of $\Mt$ jumps to $I_j+1$, while no other coordinate jumps, which yields that $\M_{T^{\ell}_{I_{\ell}+1}}=\Phi(I')=\sigma'$. Since there is no jump between $t$ and $T^{\ell}_{I_{\ell}+1}$, this implies that
    \begin{align}
        \mathbb{P}\Big(\exists t\in(0,\infty):\M_{t-}=\sigma,\M_t=\sigma'\Big)\geq\mathbb{P}(E)\,.\label{eq:bound with E}
    \end{align}
    Hence, under the assumption that $\mathbb{P}(E)>0$, it follows that $(\sigma,\sigma')\in E(\G_\M)$.
    
    We now prove that $\mathbb{P}(E)>0$. We first prove by induction that, for any subset $J\subset[n]$, we have
    \begin{align*}
        \mathbb{P}\Big(\max\big\{T^j_{I_j}:j\in J\big\}<\min\big\{T^j_{I_j+1}:j\in J\big\}\Big)>0\,.
    \end{align*}
    This holds if $|J|=1$ since $\M$ is regular, meaning that there are no two jumps at the same time. Now assume that the statement holds for all sets of size $k-1$ and let $J\subset[n]$ be of size $k$. Let $J'$ be any subset of $J$ of size $k-1$ and write $i=J\setminus J'$. By the induction hypothesis, we know that 
    \begin{align*}
        \mathbb{P}\Big(\max\big\{T^j_{I_j}:j\in J'\big\}<\min\big\{T^j_{I_j+1}:j\in J'\big\}\Big)>0\,.
    \end{align*}
    But then, since $T^i_{I_i}$ is independent of $((T^j_k)_{0\leq k\leq j})_{j\in J'}$ and is continuously distributed, we have that
    \begin{align*}
        \mathbb{P}\Big(\max\big\{T^j_{I_j}:j\in J'\big\}<T^i_{I_i}<\min\big\{T^j_{I_j+1}:j\in J'\big\}~\Big|~\big((T^j_k)_{0\leq k\leq j}\big)_{j\in J'}\Big)>0\,.
    \end{align*}
    Since we have that $T^i_{I_i}<T^i_{I_i+1}$ by assumption, it follows that
    \begin{align*}
        \mathbb{P}\Big(\max\big\{T^j_{I_j}:j\in J\big\}<\min\big\{T^j_{I_j+1}:j\in J\big\}\Big)>0\,,
    \end{align*}
    verifying the inductive step. But now, using a similar argument as the one to go from $J'$ to $J$, applied with $J'=[n]\setminus\{\ell\}$ to $I'$, we have that
    \begin{align*}
        \mathbb{P}\Big(\max\big\{T^j_{I'_j}:j\neq \ell\big\}<T^{\ell}_{I'_{\ell}}<\min\big\{T^j_{I'_j+1}:j\neq \ell\big\}~\Big|~\big((T^j_k)_{0\leq k\leq j}\big)_{j\neq \ell}\Big)>0\,.
    \end{align*}
    Since $I'_j=I_j$ for $j\neq\ell$ and $I'_\ell=I_\ell+1$, it follows that
    \begin{align*}
        \mathbb{P}\Big(\max\big\{T^j_{I_j}:j\neq \ell\big\}<T^{\ell}_{I_{\ell}+1}<\min\big\{T^j_{I_j+1}:j\neq \ell\big\}~\Big|~\big((T^j_k)_{0\leq k\leq j}\big)_{j\neq \ell}\Big)>0\,,
    \end{align*}
    which, combined with our previous result applied to $J=[n]\setminus\{\ell\}$
    \begin{align*}
        \mathbb{P}\Big(\max\big\{T^j_{I_j}:j\neq \ell\big\}<\min\big\{T^j_{I_j+1}:j\neq \ell\big\}\Big)>0\,,
    \end{align*}
    leads to $\mathbb{P}(E)>0$. Combining $\mathbb{P}(E)>0$ with \eqref{eq:bound with E} and the fact that $(\sigma,\sigma')\in E(\mathcal{H}_n)$ was arbitrary yields that $\mathcal{H}_n$ is a subgraph of $\G_\M$.
    
    We conclude this proof by proving that $\langle\mathcal{H}_n\rangle=\T$. Since we know that $\langle\G_\M\rangle\subset\T$ and that $\G_\M=\mathcal{H}_n$ which implies that $\langle\G_\M\rangle=\langle\mathcal{H}_n\rangle$, it suffices to show that $\T\subset\langle\mathcal{H}_n\rangle$. Let $\transp{i}{j}\in\T$ be such that $i<j$. We now prove that there exists $\sigma\in\Sn$ such that $(\sigma,\sigma\cdot\transp{i}{j})\in E(\mathcal{H}_n)$.
    
    Fix $I=(I_j)_{j\in[n]}\in\mathcal{E}_n$ such that $I_j=j-i-1<j-1$ and that $I_k=0$ for $k\neq j$. Let $I'=(I'_j)_{j\in[n]}\in\mathcal{E}_n$ be such that $I'_k=I_k=0$ for $k\neq j$ and that $I'_j=I_j+1=j-i\leq j-1$. From the definition of $\mathcal{H}_n$, we know that $(\Phi(I),\Phi(I'))\in E(\mathcal{H}_n)$. It now only remains to prove that $\Phi(I')=\Phi(I)\cdot\transp{i}{j}$. Using the definition of $\Phi$, in particular the constructive definition of $\Phi$ given by \eqref{eq:recursive Phi}, we have that
    \begin{align}
        \Phi(I)(k)&=\left\{\begin{array}{ll}
            k & \textrm{if $k\leq i$ or $k>j$} \\
            i+1 & \textrm{if $k=j$} \\
            k+1 & \textrm{otherwise}
        \end{array}\right.\label{eq:PhiI}\\
        &=(\textcolor{black!50!white}{1,\ldots,i-1,}\boldsymbol{i}\textcolor{black!50!white}{,i+2,i+3,\ldots,j-1,j,}\boldsymbol{i+1}\textcolor{black!50!white}{,j+1,\ldots,n})\notag
    \end{align}
    and that
    \begin{align}
        \Phi(I')(k)&=\left\{\begin{array}{ll}
            k & \textrm{if $k<i$ or $k>j$} \\
            i & \textrm{if $k=j$} \\
            k+1 & \textrm{otherwise}\,.
        \end{array}\right.\label{eq:PhiI'}\\
        &=(\textcolor{black!50!white}{1,\ldots,i-1,}\boldsymbol{i+1}\textcolor{black!50!white}{,i+2,i+3,\ldots,j-1,j,}\boldsymbol{i}\textcolor{black!50!white}{,j+1,\ldots,n})\,.\notag
    \end{align}
    Indeed, the only cause of non-zero inversion number of $\Phi(I)$ comes from the value $i+1$ being found at position $j$, creating $j-i-1$ inversions, and the only cause of non-zero inversion number of $\Phi(I')$ comes from the value $i$ being found at position $j$, creating $j-i$ inversions. Finally, \eqref{eq:PhiI} and \eqref{eq:PhiI'} imply that $\Phi(I)(k)=\Phi(I')(k)$ for any $k\notin\{i,j\}$ and that $\Phi(I')(i)=\Phi(I)(j)=i+1$ and $\Phi(I')(j)=\Phi(I)(i)=i$; this exactly means that $\Phi(I')=\Phi(I)\cdot\transp{i}{j}$, which concludes our proof.
\end{proof}

\subsection{Jumping times}\label{subsec:convergencePoisson}

In this section we prove Theorem~\ref{thm:PoissonTk}. The proof is divided into two parts: first we derive the distributional limit of a single jumping time, then we show that the full distribution is as claimed.

\begin{proof}[Proof of Theorem~\ref{thm:PoissonTk}]
    Fix $k\in\mathbb{N}$. In the case of a monotone Mallows process $\Mt$, since $\Inv(\M_t)$ is increasing, we have that
    \begin{align*}
        \mathbb{P}\big(T_k>t\big)=\mathbb{P}\big(\Inv(\M_t)\leq k-1\big)\,.
    \end{align*}
    Using the definition of Mallows permutations, it follows that
    \begin{align*}
        \mathbb{P}\big(T_k>t\big)=\frac{1}{\prod_{k=1}^n\left(\sum_{\ell=0}^{k-1}t^\ell\right)}\sum_{\ell=0}^{k-1}\Big|\Big\{\sigma\in\Sn:\Inv(\sigma)=\ell\Big\}\Big|t^\ell\,
    \end{align*}
    which means that
    \begin{align*}
        \mathbb{P}\big(nT_k>t\big)=\frac{1}{\prod_{k=1}^n\left(\sum_{\ell=0}^{k-1}\left(\frac{t}{n}\right)^\ell\right)}\sum_{\ell=0}^{k-1}\frac{\Big|\Big\{\sigma\in\Sn:\Inv(\sigma)=\ell\Big\}\Big|}{n^\ell}t^\ell\,.
    \end{align*}
    
    For the denominator, note that, for $n>t$, we have
    \begin{align*}
        \prod_{k=1}^n\left(\sum_{\ell=0}^{k-1}\left(\frac{t}{n}\right)^\ell\right)=\prod_{k=1}^n\left(\frac{1-\left(\frac{t}{n}\right)^k}{1-\frac{t}{n}}\right)=\left(1-\frac{t}{n}\right)^{-n}\prod_{k=1}^n\left(1-\left(\frac{t}{n}\right)^k\right)\,.
    \end{align*}
    Now, as $n$ tends to infinity, the first term converges to $e^t$:
    \begin{align*}
        \left(1-\frac{t}{n}\right)^{-n}\longrightarrow e^t\,.
    \end{align*}
    For the second term, we have
    \begin{align*}
        \prod_{k=1}^n\left(1-\left(\frac{t}{n}\right)^k\right)&=\exp\left(\sum_{k=1}^n\log\left(1-\left(\frac{t}{n}\right)^k\right)\right)\,.
    \end{align*}
    Since $(t/n)^k$ is decreasing in $k$, we can apply a uniform bound over all the logarithmic terms and obtain
    \begin{align}
        \prod_{k=1}^n\left(1-\left(\frac{t}{n}\right)^k\right)&=\exp\left(O\left(\frac{1}{n}\right)\right)=1+o(1)\,.\label{eq:prod goes to 1}
    \end{align}
    This proves that
    \begin{align*}
        \prod_{k=1}^n\left(\sum_{\ell=0}^{k-1}\left(\frac{t}{n}\right)^\ell\right)\longrightarrow e^t
    \end{align*}
    as $n$ tends to infinity.
    
    Consider now $S^\ell_n=|\{\sigma\in\Sn:\Inv(\sigma)=\ell\}|$. Using the bijection $\Phi$ from Section~\ref{sec:constructingMallows}, we have that
    \begin{align*}
        S^\ell_n=\Big|\Big\{(I_1,\ldots,I_n)\in\mathcal{E}_n:I_1+\ldots+I_n=\ell\Big\}\Big|\,.
    \end{align*}
    By changing the condition that $0\leq I_j\leq j-1$ in the definition of $\mathcal{E}_n$ to $I_j\geq0$ for an upper bound and to $0\leq I_j\leq\min(1,j-1)$ for a lower bound, we obtain
    \begin{align*}
        S^\ell_n\geq\Big|\Big\{(I_1,\ldots,I_n):I_1+\ldots+I_n=\ell\textrm{ and }0\leq I_j\leq\min(1,j-1)\Big\}\Big|=\binom{n-1}{\ell}
    \end{align*}
    and
    \begin{align*}
        S^\ell_n\leq\Big|\Big\{(I_1,\ldots,I_n):I_1+\ldots+I_n=\ell\textrm{ and }I_j\geq0\Big\}\Big|=\binom{n+\ell}{\ell}\,.
    \end{align*}
    It follows that, for any fixed $\ell$ and as $n$ goes to infinity, we have
    \begin{align*}
        \frac{S^\ell_n}{n^\ell}\longrightarrow\frac{1}{\ell!}\,.
    \end{align*}
    Since $k$ is fixed, this implies that
    \begin{align*}
        \sum_{\ell=0}^{k-1}\frac{\Big|\Big\{\sigma\in\Sn:\Inv(\sigma)=\ell\Big\}\Big|}{n^\ell}t^\ell\longrightarrow\sum_{\ell=0}^{k-1}\frac{t^\ell}{\ell!}
    \end{align*}
    
    Combining the previous results, when $n$ tends to infinity, we have that
    \begin{align*}
        \mathbb{P}\big(nT_k>t\big)\longrightarrow e^{-t}\sum_{\ell=0}^{k-1}\frac{t^\ell}{\ell!}=\int_t^\infty\frac{u^{k-1}e^{-u}}{(k-1)!}du\,,
    \end{align*}
    which proves that $nT_k\overset{d}{\rightarrow}\mathrm{Gamma}(k,1)$.
    
    We now prove the second part of the proposition, stating that the first jumping times of regular Mallows processes converge to a Poisson point process. Let $t\in[0,\infty)$ and $s>0$. Also, for any $\sigma\in\Sn$ let $I=(I_j)_{j\in[n]}=\Phi^{-1}(\sigma)$. Since the process $\Mt$ is regular, it has independent inversions, which implies that
    \begin{align}
        &\hspace{-0.5cm}\mathbb{P}\Big(\M_\frac{t+s}{n}=\sigma~\Big|~\M_\frac{t}{n}=\sigma,(\M_u)_{0\leq u<\frac{t}{n}}\Big)\label{eq:dependency}\\
        &=\mathbb{P}\Big(\forall j\in[n],\,\Inv_j(\M_\frac{t+s}{n})=I_j~\Big|~\forall j\in[n],\,\Inv_j(\M_\frac{t}{n})=I_j,\big(\Inv_j(\M_u)\big)_{0\leq u<\frac{t}{n}}\Big)\notag\\
        &=\prod_{j=1}^n\mathbb{P}\Big(\Inv_j(\M_\frac{t+s}{n})=I_j~\Big|~\Inv_j(\M_\frac{t}{n})=I_j,\big(\Inv_j(\M_u)\big)_{0\leq u<\frac{t}{n}}\Big)\,.\notag
    \end{align}
    Note that, whatever the value of $I_j\geq0$ is, we have that
    \begin{align*}
        &\hspace{-0.5cm}\mathbb{P}\Big(\Inv_j(\M_\frac{t+s}{n})=I_j~\Big|~\Inv_j(\M_\frac{t}{n})=I_j,\big(\Inv_j(\M_u)\big)_{0\leq u<\frac{t}{n}}\Big)\\
        &=1-\mathbb{P}\Big(\Inv_j(\M_\frac{t+s}{n})>I_j~\Big|~\Inv_j(\M_\frac{t}{n})=I_j,\big(\Inv_j(\M_u)\big)_{0\leq u<\frac{t}{n}}\Big)
    \end{align*}
    Moreover, by definition, we have
    \begin{align*}
        &\hspace{-0.5cm}\mathbb{E}\bigg[\mathbb{P}\Big(\Inv_j(\M_\frac{t+s}{n})>I_j~\Big|~\Inv_j(\M_\frac{t}{n})=I_j,\big(\Inv_j(\M_u)\big)_{0\leq u<\frac{t}{n}}\Big)\bigg]\\
        &=\mathbb{P}\Big(\Inv_j(\M_\frac{t+s}{n})>I_j~\Big|~\Inv_j(\M_\frac{t}{n})=I_j\Big)\\&\leq\frac{\mathbb{P}\big(\Inv_j(\M_\frac{t+s}{n})>I_j\big)}{\mathbb{P}\big(\Inv_j(\M_\frac{t}{n})=I_j\big)}\,,
    \end{align*}
    and by the definition of $\Mt$, it follows that
    \begin{align*}
        \frac{\mathbb{P}\big(\Inv_j(\M_\frac{t+s}{n})>I_j\big)}{\mathbb{P}\big(\Inv_j(\M_\frac{t}{n})=I_j\big)}=\frac{1-\frac{t}{n}}{\left(\frac{t}{n}\right)^{I_j}-\left(\frac{t}{n}\right)^j}\cdot\frac{\left(\frac{t+s}{n}\right)^{I_j+1}-\left(\frac{t+s}{n}\right)^j}{1-\frac{t+s}{n}}\underset{n\rightarrow\infty}{\longrightarrow}0\,.
    \end{align*}
    The previous results imply that
    \begin{align}
        \mathbb{P}\Big(\Inv_j(\M_\frac{t+s}{n})=I_j~\Big|~\Inv_j(\M_\frac{t}{n})=I_j,\big(\Inv_j(\M_u)\big)_{0\leq u<\frac{t}{n}}\Big)\longrightarrow1\label{eq:convTo1}
    \end{align}
    in probability, as $n$ goes to infinity. Now, recall that, by assumption, for any $j\in[n]$, the process $(\Inv_j(\M_t))_{t\in[0,\infty)}$ is increasing. This implies that, for $I_j=0$, we have
    \begin{align}
        &\hspace{-0.5cm}\mathbb{P}\Big(\Inv_j(\M_\frac{t+s}{n})=0~\Big|~\Inv_j(\M_\frac{t}{n})=0,\big(\Inv_j(\M_u)\big)_{0\leq u<\frac{t}{n}}\Big)\label{eq:0case}\\
        &=\mathbb{P}\Big(\Inv_j(\M_\frac{t+s}{n})=0~\Big|~\Inv_j(\M_\frac{t}{n})=0\Big)\notag\\
        &=\frac{\mathbb{P}\big(\Inv_j(\M_\frac{t+s}{n})=0\big)}{\mathbb{P}\big(\Inv_j(\M_\frac{t}{n})=0\big)}\notag\\
        &=\frac{1-\frac{t+s}{n}}{1-\left(\frac{t+s}{n}\right)^j}\cdot\frac{1-\left(\frac{t}{n}\right)^j}{1-\frac{t}{n}}\,.\notag
    \end{align}

    We will now combine \eqref{eq:convTo1} and \eqref{eq:0case} with \eqref{eq:dependency} to obtain the convergence of the sequence of the first jumping times of $\Mt$. Fix $k\geq0$. For any $\sigma\in\Sn$ such that $\Inv(\sigma)\leq k$, by separating according to whether $I_j=0$ or $I_j>0$, \eqref{eq:dependency} and \eqref{eq:0case} tell us that
    \begin{align*}
        &\hspace{-0.5cm}\mathbb{P}\Big(\M_\frac{t+s}{n}=\sigma~\Big|~\M_\frac{t}{n}=\sigma,(\M_u)_{0\leq u<\frac{t}{n}}\Big)\\
        &=\prod_{j:I_j>0}\mathbb{P}\Big(\Inv_j(\M_\frac{t+s}{n})=I_j~\Big|~\Inv_j(\M_\frac{t}{n})=I_j,\big(\Inv_j(\M_u)\big)_{0\leq u<\frac{t}{n}}\Big)\\
        &\hspace{0.5cm}\times\prod_{j:I_j=0}\frac{1-\frac{t+s}{n}}{1-\left(\frac{t+s}{n}\right)^j}\cdot\frac{1-\left(\frac{t}{n}\right)^j}{1-\frac{t}{n}}\,.
    \end{align*}
    Note that $|\{j:I_j>0\}|\leq\Inv(\sigma)\leq k$. Combining this fact with the result of \eqref{eq:convTo1} and the assumption that the different inversion numbers are independent processes, it follows that
    \begin{align*}
        \prod_{j:I_j>0}\frac{\mathbb{P}\Big(\Inv_j(\M_\frac{t+s}{n})=I_j~\Big|~\Inv_j(\M_\frac{t}{n})=I_j,\big(\Inv_j(\M_u)\big)_{0\leq u<\frac{t}{n}}\Big)}{\frac{1-\frac{t+s}{n}}{1-\left(\frac{t+s}{n}\right)^j}\cdot\frac{1-\left(\frac{t}{n}\right)^j}{1-\frac{t}{n}}}=1+o_\mathbb{P}(1)\,.
    \end{align*}
    This means that the previous result is
    \begin{align*}
        \mathbb{P}\Big(\M_\frac{t+s}{n}=\sigma~\Big|~\M_\frac{t}{n}=\sigma,(\M_u)_{u<\frac{t}{n}}\Big)&=\big(1+o_\mathbb{P}(1)\big)\prod_{j=1}^n\frac{1-\frac{t+s}{n}}{1-\left(\frac{t+s}{n}\right)^j}\cdot\frac{1-\left(\frac{t}{n}\right)^j}{1-\frac{t}{n}}
    \end{align*}
    as $n\rightarrow\infty$, where for any fixed $k\in\mathbb{N}$ the $o_\mathbb{P}(1)$ term is uniform over all permutations with $\Inv(\sigma)\leq k$. Using the convergence result stated in \eqref{eq:prod goes to 1}, we know that
    \begin{align*}
        \prod_{j=1}^n\frac{1-\left(\frac{t}{n}\right)^j}{1-\left(\frac{t+s}{n}\right)^j}\longrightarrow1\,,
    \end{align*}
    from which it follows that, for any $k\in\mathbb{N}$, uniformly over permutations $\sigma$ with $\Inv(\sigma)\leq k$,
    \begin{align*}
        \mathbb{P}\Big(\M_\frac{t+s}{n}=\sigma~\Big|~\M_\frac{t}{n}=\sigma,(\M_u)_{u<\frac{t}{n}}\Big)&=\big(1+o_\mathbb{P}(1)\big)\left(\frac{1-\frac{t+s}{n}}{1-\frac{t}{n}}\right)^n=\big(1+o_\mathbb{P}(1)\big)e^{-s}\,.
    \end{align*}
    This exactly corresponds to saying that the first inter jumping times of the process are asymptotically distributed as independent $\mathrm{Exponential}(1)$ random variables, proving the desired result.
\end{proof}

\subsection{Properties of the uniform Mallows processes}\label{subsec:propertiesUniform}

In this section we consider the uniform Mallows process $\Xt{\M^U}$ defined in Section~\ref{subsec:uniformMallows}. For the rest of this section, drop the superscript $U$ on $\M^U$ and write UMP instead of uniform Mallows process.

Recall from Definition~\ref{def:uniformModel} that the UMP is defined using a set of $n$ uniform random variables. This means that this process is generated from a well-controlled amount of randomness and it is naturally interesting to wonder when all the information can be retrieved from the evolution of the process.

Given an occurrence of the UMP $\Mt$ and a random time $T$, say that $U_j$ is \textit{retrievable} with respect to $\Mt$ and $T$ if the random variable $U_j$ is measurable with respect to the $\sigma$-algebra $\sigma(T,(\M_t)_{t\in[0,T]})$. Note that, once one of the component processes $(I_j(t))_{t\in[0,\infty)}$ jumps, it is possible to deduce the value of $U_j$ that generated this process (except for $j=1$ since $I_1(t)=0$ for all $t$). This remark implies that $U_2,\ldots,U_j$ are retrievable with respect to $\Mt$ and $T_{\binom{n}{2}}=\inf\{t>0:\Inv(\M_t)=\binom{n}{2}\}$. We are now interested in answering the following two questions:
\begin{itemize}
    \item What is the minimal time such that $U_2,\ldots,U_n$ are all retrievable?
    \item Given a deterministic time $t\in[0,\infty)$, how many of the variables $U_2,\ldots,U_n$ are retrievable?
\end{itemize}
We answer these two questions in the next two sections.

\subsubsection{Time of full information}

The following proposition describes the minimal full retrieval time of the UMP.

\begin{prop}\label{prop:retriveUniforms}
    Let $\Mt$ be the UMP. Then there exists an $\M_t$-stopping time $T_U$ such that $U_2,\ldots,U_n$ are retrievable with respect to $\Mt$ and $T_U$, and such that, for any $T$ such that $T<T_U$, not all $U_2,\ldots,U_n$ are retrievable with respect to $\Mt$ and $T$. Moreover, there exists a subset $A_U\subset\Sn$ such that
    \begin{align*}
        T_U=\inf\Big\{t\in[0,\infty):\M_t\in A_U\Big\}\,.
    \end{align*}
    Finally, we have that
    \begin{align*}
        \lim_{t\rightarrow\infty}t\mathbb{P}(T_U>t)=1\,,
    \end{align*}
    which implies that $\mathbb{E}[T_U]=\infty$.
\end{prop}

\begin{proof}
    By using that $U_i$ is retrievable if and only if $f_i(t;U_i)\geq1$, the definition of $A_U$ simply corresponds to
    \begin{align*}
        A_U&=\Big\{\sigma:\Phi^{-1}(\sigma)_j\geq1,\,\forall 2\leq j\leq n\Big\}\\
        &=\Phi\Big(\big\{I=(I_1,\ldots,I_n)\in\mathcal{E}_n:I_j\geq1,\,\forall2\leq j\leq n\big\}\Big)\,.
    \end{align*}
    This proves the first two statements of the proposition.
    
    For the third and last statement, by the definition of $A_U$ we have
    \begin{align*}
        \mathbb{P}\big(T_U<t\big)&=\mathbb{P}\Big(f_j(t;U_j)\geq1,\forall2\leq j\leq n\Big)\\
        &=\prod_{j=2}^n\left(1-\frac{1}{\sum_{\ell=0}^{j-1}t^\ell}\right)\,.
    \end{align*}
    It follows that, as $t$ tends to infinity (with $n$ being fixed), we have
    \begin{align*}
        \mathbb{P}\big(T_U<t\big)=\prod_{j=2}^n\left(1-(1+o(1))\frac{1}{t^{j-1}}\right)=1-\frac{1}{t}+o\left(\frac{1}{t}\right)\,.
    \end{align*}
    This proves that $t\mathbb{P}(T_U>t)=1+o(1)$ and concludes the proof of this proposition.
\end{proof}

\subsubsection{Amount of information}

In the previous section, we consider when does the UMP become completely deterministic and we show that there exists a specific time at which we have all the information on this process. However, we also prove that the expected time to have all the information on this process is infinite. In this section, we consider the converse approach and ask what amount of information we have access to by a given time.

Fix $t\in(0,\infty)$ and let $\I_t$ be the amount of information of the UMP at time $t$ defined by
\begin{align*}
    \I_t=\frac{1}{n}\times\Big|\Big\{j\in[n]:U_j\textrm{ is retrievable with respect to $\Mt$ and $t$}\Big\}\Big|\,.
\end{align*}
The fraction $\frac{1}{n}$ is used here so that $\I_t\in[0,1]$. The following proposition characterizes the behaviour of $\I_t$ for different regimes and, in particular, allows us to see a strong change of behaviour at $t=1$.

\begin{thm}\label{thm:convergenceIt}
    For any fixed $t\in[0,\infty)$, as $n$ tends to infinity, we have
    \begin{align*}
        \I_t\longrightarrow t\wedge 1\,,
    \end{align*}
    almost surely. Moreover, its fluctuations around its asymptotic value are given by the following formulas. For $t\in[0,1)$, we have that
    \begin{align*}
        \big(\sqrt{n}(\mathcal{I}_t-t)\big)_{t\in[0,1)}\overset{d}{\longrightarrow}(W_t)_{t\in[0,1)}\,,
    \end{align*}
    where $(W_t)_{t\in[0,1)}$ is the restriction to $[0,1)$ of a standard Brownian bridge and the convergence is with respect to the topology of uniform convergence on $\mathcal{C}([0,1))$. For $t=1$, the following central limit theorem holds
    \begin{align*}
        \frac{n}{\sqrt{\log n}}\left(\I_1-1+\frac{\log n}{n}\right)\overset{d}{\longrightarrow}\mathrm{Normal}(0,1)\,.
    \end{align*}
    Finally, for $t>1$, then
    \begin{align*}
        \big(n(1-\I_t)\big)_{t>1}\overset{d}{\longrightarrow}(X_t)_{t>1}\,,
    \end{align*}
    where $(X_t)_{t>1}$ is a decreasing process with moment generating function
    \begin{align*}
        \mathbb{E}[u^{X_t}]=\prod_{k=1}^\infty\left(1+(u-1)\frac{t-1}{t^k-1}\right)\,.
    \end{align*}
\end{thm}

\begin{proof}
    All the results of this proposition will strongly depend on the decomposition
    \begin{align*}
        \mathcal{I}_t=\frac{1}{n}\sum_{j=1}^n X_j(t)\,,
    \end{align*}
    where $X_j(t)$ is the indicator that $U_j$ is retrievable from $\Mt$ and $t$. Note that the different random variables $(X_j(t))_{j\in[n]}$ are independent Bernoulli random variables with
    \begin{align}
        \mathbb{P}\big(X_j(t)=0\big)=1-\mathbb{P}\big(X_j(t)=1\big)=\frac{1}{\sum_{\ell=0}^{j-1}t^\ell}\,.\label{eq:Xj}
    \end{align}

    For any $t\in[0,\infty)$, since $(X_j(t))_{j\in[n]}$ are all bounded by $1$, the strong law of large numbers applies~\cite{etemadi1983laws} and yields that $\mathcal{I}_t$ almost surely converges to its asymptotic expected value. Furthermore, we have that
    \begin{align*}
        \mathbb{E}[\mathcal{I}_t]&=\frac{1}{n}\sum_{j=1}^n\mathbb{E}[X_j(t)]=\frac{1}{n}\sum_{j=1}^n\left[1-\frac{1}{\sum_{\ell=0}^{j-1}t^\ell}\right]\,.
    \end{align*}
    For $t=1$, this gives us that
    \begin{align}
        \mathbb{E}[\mathcal{I}_1]=1-\frac{1}{n}\sum_{j=1}^n\frac{1}{j}=1-\frac{\log n}{n}+O\left(\frac{1}{n}\right)=1-o(1)\,,\label{eq:E(I1)}
    \end{align}
    Moreover, since $\Xt{\mathcal{I}}$ is increasing in $t$ and takes values in $[0,1]$, it follows that $\mathbb{E}[\mathcal{I}_t]=1-o(1)$ for $t>1$. Finally, for $t<1$, we have
    \begin{align*}
        \mathbb{E}[\mathcal{I}_t]=\frac{1}{n}\sum_{j=1}^n\left[1-\frac{1-t}{1-t^j}\right]=\frac{1}{n}\left[nt-\sum_{j=1}^n\frac{t^j(1-t)}{1-t^j}\right]=t+O\left(\frac{1}{n}\right)=t+o(1)\,;
    \end{align*}
    together, these facts establish the first statement of the proposition.
    
    For the second statement, we show that, for any $s<1$, $(\sqrt{n}(\mathcal{I}_t-t))_{t\in[0,s]}$ converges in distribution to the restriction to $[0,s]$ of a Brownian bridge. Fix $s\in(0,1)$. First of all, using that $t\mapsto X_j(t)$ is increasing, along with \eqref{eq:Xj}, it follows that \begin{align*}
        (X_j(t))_{t\in[0,s]}\overset{d}{=}\left(\min\left\{1,\left\lfloor U\frac{1-t^j}{1-t}\right\rfloor\right\}\right)_{t\in[0,s]}\,,
    \end{align*}
    where $U$ is a $\mathrm{Uniform}([0,1])$ random variable. Since the variables $(X_j(t))_{t\in[0,s]}$ are independent over $j\in[n]$, fixing a set of $n$ independent $\mathrm{Uniform}([0,1])$ random variables $(\Tilde{U}_1,\ldots,\Tilde{U}_n)$, we then have
    \begin{align}
        (\mathcal{I}_t)_{t\in[0,s]}\overset{d}{=}(\Tilde{\mathcal{I}}_t)_{t\in[0,s]}:=\left(\frac{1}{n}\sum_{j=1}^n\min\left\{1,\left\lfloor\Tilde{U}_j\frac{1-t^j}{1-t}\right\rfloor\right\}\right)_{t\in[0,s]}\,.\label{eq:def I}
    \end{align}
    The relation between $(U_1,\ldots,U_n)$ in the definition of the UMP and $(\Tilde{U}_1,\ldots,\Tilde{U}_n)$ here does not need to be specified; however it is worth noting that they are not the same variables. Consider now the process
    \begin{align}
        (\Tilde{\mathcal{J}}_t)_{t\in[0,s]}:=\left(\frac{1}{n}\sum_{j=1}^n\min\left\{1,\left\lfloor\Tilde{U}_j\frac{1}{1-t}\right\rfloor\right\}\right)_{t\in[0,s]}\,.\label{eq:def J}
    \end{align}
    and note that, since $t\leq s<1$, we have
    \begin{align*}
        \min\left\{1,\left\lfloor\Tilde{U}_j\frac{1}{1-t}\right\rfloor\right\}=\mathbbm{1}_{1-U_j\leq t}\,.
    \end{align*}
    This means that $(\Tilde{\mathcal{J}}_t)_{t\in[0,s]}$ is the empirical distribution function of $(1-\Tilde{U}_1,\ldots,1-\Tilde{U_n})$ restricted to $[0,s]$, and from~\cite[Corollary~20.14]{aldous1985exchangeability}, we know that $(\sqrt{n}(\Tilde{\mathcal{J}}_t-t))_{t\in[0,s]}$ converges to the restriction to $[0,s]$ of a standard Brownian bridge (in fact the convergence occurs on $[0,1]$, but we are only interested in its behaviour on $[0,s]$). Hence, it only remains to relate $\Tilde{\mathcal{I}}$ and $\Tilde{\mathcal{J}}$.
    
    Using that $\frac{1-t^j}{1-t}\leq\frac{1}{1-t}$, we have
    \begin{align*}
        &\min\left\{1,\left\lfloor\Tilde{U}_j\frac{1-t^j}{1-t}\right\rfloor\right\}\neq\min\left\{1,\left\lfloor\Tilde{U}_j\frac{1}{1-t}\right\rfloor\right\}\\
        \Longleftrightarrow~&\min\left\{1,\left\lfloor\Tilde{U}_j\frac{1-t^j}{1-t}\right\rfloor\right\}=0\textrm{ and }\min\left\{1,\left\lfloor\Tilde{U}_j\frac{1}{1-t}\right\rfloor\right\}=1\\
        \Longleftrightarrow~&\Tilde{U}_j\frac{1-t^j}{1-t}<1\textrm{ and }\Tilde{U}_j\frac{1}{1-t}\geq1\\
        \Longleftrightarrow~&\Tilde{U}_j\in\left[1-t,\frac{1-t}{1-t^j}\right)\,.
    \end{align*}
    Now, let $(N_t)_{t\in[0,s]}$ be the random process defined by
    \begin{align*}
        N_t:&=\left|\left\{j\in[n]:\min\left\{1,\left\lfloor\Tilde{U}_j\frac{1-t^j}{1-t}\right\rfloor\right\}\neq\min\left\{1,\left\lfloor\Tilde{U}_j\frac{1}{1-t}\right\rfloor\right\}\right\}\right|\\
        &=\left|\left\{j\in[n]:\Tilde{U}_j\in\left[1-t,\frac{1-t}{1-t^j}\right)\right\}\right|
    \end{align*}
    and let $N:=\max_{t\in[0,s]}N_t$. The definition of $N_t$ along with the definition of $(\Tilde{\mathcal{I}}_t)_{t\in[0,s]}$ and $(\Tilde{\mathcal{J}}_t)_{t\in[0,s]}$ in \eqref{eq:def I} and \eqref{eq:def J} immediately imply that
    \begin{align}
        \big|\Tilde{\mathcal{I}}_t-\Tilde{\mathcal{J}}_t\big|\leq\frac{N_t}{n}\,,\label{eq:IJ and N}
    \end{align}
    which also means that $N$ is a uniform bound over the difference between the two processes $(\Tilde{\mathcal{I}}_t)_{t\in[0,s]}$ and $(\Tilde{\mathcal{J}}_t)_{t\in[0,s]}$. We now provide an upper tail bounds on $N$.
    
    Before providing the upper tail bound on $N$, we claim that
    \begin{align}
        N=\max_{i\in[n]:1-\Tilde{U}_i\leq s}\big\{N_{1-\Tilde{U}_i}\big\}\,.\label{eq:N and U}
    \end{align}
    To see this, fix $t\in[0,s]$ and let $\delta>0$ be such that $[1-(t+\delta),1-t)\cap\{\Tilde{U}_1,\ldots,\Tilde{U}_n\}=\emptyset$. Since the map $u\mapsto(1-u)/(1-u^j)$ is decreasing for any $j\in[n]$, if $\Tilde{U}_j\in[1-(t+\delta),(1-(t+\delta))/(1-(t+\delta)^j))$, then $\Tilde{U}_j\in[1-t,(1-t)/(1-t^j)]$. Thus
    \begin{align*}
        \left\{j\in[n]:\Tilde{U}_j\in\left[1-(t+\delta),\frac{1-(t+\delta)}{1-(t+\delta)^j}\right)\right\}\subseteq\left\{j\in[n]:\Tilde{U}_j\in\left[1-t,\frac{1-t}{1-t^j}\right)\right\}
    \end{align*}
    and so $N_{t+\delta}\leq N_t$. Since $t\in[0,s]$ was arbitrary, this proves that $t\mapsto N_t$ only has increments at times $\{1-\Tilde{U}_1,\ldots,1-\Tilde{U}_n\}$, establishing \eqref{eq:N and U}.
    
    Thanks to \eqref{eq:N and U} and the definition of $N_t$, we now have
    \begin{align*}
        N&=\max_{i\in[n]:1-\Tilde{U}_i\leq s}\left|\left\{j\in[n]:\Tilde{U}_j\in\left[\Tilde{U}_i,\frac{\Tilde{U}_i}{1-(1-\Tilde{U}_i)^j}\right)\right\}\right|\,.
    \end{align*}
    Using that
    \begin{align*}
        \frac{1-t}{1-t^j}=1-t+\frac{t^j(1-t)}{1-t^j}\leq 1-t+s^j
    \end{align*}
    for any $t\leq s<1$, it follows that
    \begin{align*}
        \frac{\Tilde{U}_i}{1-(1-\Tilde{U}_i)^j}\leq\Tilde{U}_i+s^j
    \end{align*}
    for any $1-\Tilde{U}_i\leq s$, and then
    \begin{align*}
        N&\leq\max_{i\in[n]:1-\Tilde{U}_i\leq s}\left|\Big\{j\in[n]:\Tilde{U}_j\in\big[\Tilde{U}_i,\Tilde{U}_i+s^j\big)\Big\}\right|\leq\max_{i\in[n]}\left|\Big\{j\in[n]:\Tilde{U}_j\in\big[\Tilde{U}_i,\Tilde{U}_i+s^j\big)\Big\}\right|\,.
    \end{align*}
    Fix now $k\in\mathbb{N}$ and use the previous bound to see that
    \begin{align}
        \mathbb{P}(N\geq k)&\leq\sum_{i=1}^n\mathbb{P}\left(\sum_{j\in[n]:j\neq i}\mathbbm{1}_{\{\Tilde{U}_j\in[U_i,U_i+s^j)\}}\geq k-1\right)\label{eq:bound N and k}\\
        &=\sum_{i=1}^n\mathbb{P}\left(\sum_{j\in[n]:j\neq i}\mathbbm{1}_{\{\Tilde{U}_j<s^j\}}\geq k-1\right)\,,\notag
    \end{align}
    where the second inequality uses that the variables $(\Tilde{U}_j)_{j\in[n]}$ are independent $\mathrm{Uniform}([0,1])$. Moreover, for any $i\in[n]$, by Markov's inequality we have
    \begin{align}
        \mathbb{P}\left(\sum_{j\in[n]:j\neq i}\mathbbm{1}_{\{\Tilde{U}_j<s^j\}}\geq k-1\right)&=\mathbb{P}\left(\exp\left(\sum_{j\in[n]:j\neq i}\mathbbm{1}_{\{\Tilde{U}_j<s^j\}}\right)\geq e^{k-1}\right)\label{eq:sum Ui}\\
        &\leq e^{-k+1}\mathbb{E}\left[\exp\left(\sum_{j\in[n]:j\neq i}\mathbbm{1}_{\{\Tilde{U}_j<s^j\}}\right)\right]\notag\\
        &=e^{-k+1}\prod_{j\in[n]:j\neq i}\Big(1+(e-1)s^j\Big)\notag\\
        &\leq\exp\left(-k+1+(e-1)\frac{1}{1-s}\right)\,,\notag
    \end{align}
    where the last inequality follows from convexity of the exponential and by taking the product over all $j\geq0$ rather than over $j\in[n]\setminus\{i\}$. Combining \eqref{eq:bound N and k} and \eqref{eq:sum Ui}, we obtain
    \begin{align}
        \mathbb{P}(N\geq k)\leq\exp\left(\log n-k+\left[1+(e-1)\frac{1}{1-s}\right]\right)\,.\label{eq:bound N}
    \end{align}
    which we now use to conclude the convergence of $(\sqrt{n}(\mathcal{I}_t-t))_{t\in[0,s]}$ to the Brownian bridge.

    Using the definition of $\Tilde{\mathcal{I}}$ and $\Tilde{\mathcal{J}}$ from \eqref{eq:def I} and \eqref{eq:def J} along with \eqref{eq:IJ and N}, we have that
    \begin{align*}
        \sup_{t\in[0,s]}\big|\Tilde{\mathcal{I}}_s-\Tilde{\mathcal{J}}_t\big|=\frac{1}{n}N\,,
    \end{align*}
    from which \eqref{eq:bound N} tells us that
    \begin{align*}
        \mathbb{P}\left(\sup_{t\in[0,s]}\big|\Tilde{\mathcal{I}}_t-\Tilde{\mathcal{J}}_t\big|\geq\frac{2\log n}{n}\right)\leq\frac{C_s}{n}\,,
    \end{align*}
    where $C_s=\exp(1+(e-1)/(1-s))$. Moreover, we already showed that $(\sqrt{n}(\Tilde{\mathcal{J}}_t-t))_{t\in[0,s]}$ converges to the restriction to $[0,s]$ of a standard Brownian bridge. The desired convergence now simply follows since $(\mathcal{I}_t)_{t\in[0,s]}\overset{d}{=}(\Tilde{\mathcal{I}}_t)_{t\in[0,s]}$.

    Consider now the case $t=1$. We already showed that $\mathbb{E}[\mathcal{I}_1]=1-o(1)$. Moreover, we have
    \begin{align*}
        \mathrm{Var}\left[\mathcal{I}_1\right]=\frac{1}{n^2}\mathrm{Var}\left[\sum_{j=1}^nX_j(1)\right]=\frac{1}{n^2}\sum_{j=1}^n\mathrm{Var}[X_j(1)]\,,
    \end{align*}
    and using that $X_j(1)$ is a $\mathrm{Bernoulli}((j-1)/j)$, it follows that
    \begin{align*}
        \mathrm{Var}\left[\mathcal{I}_1\right]=\frac{1}{n^2}\sum_{j=1}^n\frac{j-1}{j^2}=\big(1+o(1)\big)\frac{\log n}{n^2}
    \end{align*}
    Since $X_j(t)\leq 1$, the conditions of the Lindeberg-Feller theorem~\cite[Theorem~3.4.10]{durrett2019probability} are met, and it follows that
    \begin{align*}
        \frac{\mathcal{I}_1-\mathbb{E}[\mathcal{I}_1]}{\sqrt{\mathrm{Var}(\mathcal{I}_1)}}\overset{d}{\longrightarrow}\mathrm{Normal}(0,1)\,.
    \end{align*}
    This convergence combined with the previous asymptotic approximation of $\mathrm{Var}[\mathcal{I}_1]$ and \eqref{eq:E(I1)} exactly corresponds to the desired result.
    
    For the last assertion of the proposition, using once again the distribution of $X_j(t)$, for $t>1$ we have that
    \begin{align*}
        \mathbb{E}\left[e^{iun(1-\mathcal{I}_t)}\right]&=e^{iun}\prod_{j=1}^n\left[1+(e^{-iu}-1)\left(1-\frac{t-1}{t^j-1}\right)\right]\\
        &=\prod_{j=1}^n\left[e^{iu}+(1-e^{iu})\left(1-\frac{t-1}{t^j-1}\right)\right]\\
        &=\prod_{j=1}^n\left[1+(e^{iu}-1)\frac{t-1}{t^j-1}\right]\,,
    \end{align*}
    and the final product converges to the desired function as $n$ goes to infinity. To conclude this proof, simply note that $(\mathcal{I}_t)_{t>1}$ is increasing to see that $(X_t)_{t>1}$ is decreasing, corresponding to the last statement of the proposition.
\end{proof}

\section{Conclusion and open questions}\label{sec:conclusion}

This work introduces Mallows processes and describes two natural ways to generate regular Mallows processes. However, many properties of these processes remain to be proven; we now state some possible directions for future work on Mallows processes and their properties.

\subsection{Graph structure}

Theorem~\ref{thm:MallowsGraph} states that $\Ta\subseteq\langle\G_\M\rangle\subseteq\T$ for any smooth Mallows process $\M$, and that $\langle\G_\M\rangle=\T$ whenever $\M$ is regular. The proof of the second statement strongly relies on the independent inversion property of $\M$ when it is regular, and we do not necessarily expect to have $\langle\G_\M\rangle=\T$ if we drop this assumption.

\begin{conj}
    For any $n\geq3$, there exists a smooth Mallows process $\M=\Mt$ such that
    \begin{align*}
        \langle\G_\M\rangle\subsetneq\T\,.
    \end{align*}
\end{conj}

Furthermore, when $n=3$ (the first non-trivial case), it is not hard to see that it is even possible to define $\M$ so that $\langle\G_\M\rangle=\Ta$. Whether this result can be generalized to any $n\geq3$ or not is unclear, leading to the following question.

\begin{question}
    What is the value $N\in\mathbb{N}\cup\{\infty\}$ such that, for any $n<N$, there exists a smooth Mallows process $\M=\Mt$ taking values in $\Sn$ with
    \begin{align*}
        \langle\G_\M\rangle=\Ta\,,
    \end{align*}
    and for any $n\geq N$, any smooth Mallows process $\M=\Mt$ taking values in $\Sn$ has the property that
    \begin{align*}
        \Ta\subsetneq\langle\G_\M\rangle\,.
    \end{align*}
\end{question}

Note that, in this question, we allow for $N=\infty$, which would mean that for any $n$, there exists a smooth Mallows process $\M$ with $\langle\G_\M\rangle=\Ta$. The fact that such $N$ exists can be verified as follows. Fix $n\in\mathbb{N}$ and consider a smooth Mallows process $\M=\Mt$ taking values in $\Sn$ such that $\langle\G_\M\rangle=\Ta$. Then define a Mallows process $\Tilde{\M}=\Xt{\Tilde{\M}}$ taking values in $\mathcal{S}_{n-1}$ by
\begin{align*}
    \Tilde{\M}_t(i)=\left\{\begin{array}{ll}
        \M_t(i) & \textrm{if $i<\M_t^{-1}(n)$} \\
        \M_t(i+1) & \textrm{if $\M_t^{-1}(n)\leq i<n-1$}
    \end{array}\right.\,.
\end{align*}
In other words, if $\M_t=(s_1,\ldots,s_{j-1},n,s_{j+1},\ldots,s_n)$, then $\Tilde{\M}_t=(s_1,\ldots,s_{j-1},s_{j+1},\ldots,s_n)$, which simply corresponds to removing the value $n$ in the permutation $\M_t$. It is then easy to verify that $\Xt{\Tilde{\M}}$ is a smooth Mallows process with $\langle\G_{\Tilde{\M}}\rangle=\Ta$.

\subsection{Markov processes and Markov chains}

The birth Mallows process is the unique regular Mallows process that is also a Markov process. A natural question regarding this process is whether its corresponding \textit{jumping process} $(\Tilde{\M}_k)_{0\leq k\leq\binom{n}{2}}=(\M_{T_k})_{0\leq k\leq\binom{n}{2}}$ is a Markov chain. Since the times of jump of the birth Mallows process give strong indications regarding the likelihood of follow-up jumps, we do not expect this jumping process to have the Markov property. Computations for small values of $n$ seem to confirm our conjecture but we do not have a formal proof.

\begin{conj}
    For any $n\geq4$, the jumping process of the birth Mallows process is not a Markov chain.
\end{conj}

Assuming this conjecture holds, by the uniqueness of the birth Mallows process, it is impossible to have a regular Mallows process which is a Markov process and whose jumping process is a Markov chain. However, the existence of Mallows processes that are Markov processes and whose jumping processes are Markov chains remains open if we drop some assumptions. For example, computations on small values of $n$ seem to show that there exist smooth (but not regular) Mallows processes with this property. Moreover, computations led us to believe that such processes could be chosen with interesting graph structure.

\begin{conj}
    For any $n\geq4$, there exist smooth Mallows processes $\M=\Mt$ such that $\Mt$ is a Markov process and such that $(\Tilde{\M}_k)_{0\leq k\leq\binom{n}{2}}$ is a Markov chain. Moreover, any such process has the property that
    \begin{align*}
        \Ta\subsetneq\langle\G_\M\rangle\subseteq\T\,.
    \end{align*}
    Finally, there exist such processes with the property that $\langle\G_\M\rangle\subsetneq\T$.
\end{conj}

\subsection{Another Gaussian process}

In Section~\ref{subsec:uniformMallows}, we studied properties of the uniform Mallows process and showed in Theorem~\ref{thm:convergenceIt} that there was an embedded, asymptotically Gaussian process related to the amount of information of this process (more precisely a Brownian bridge).

In parallel, the birth Mallows process also appears to have a natural associated asymptotically Gaussian process, obtained by an appropriate normalization of the number of jumps. More precisely, let the \textit{number of jumps} of $\Mt$ be the process $\Xt{J}$ defined by $J_t=|\{k\leq\binom{n}{2}:T_k\leq t\}|$. Then, by applying the Lindeberg-Feller theorem~\cite[Theorem~3.4.10]{durrett2019probability}, it is straightforward to prove the following fact.

\begin{fact}
    As $n$ goes to infinity, for any $t\in[0,1)$, we have that
    \begin{align*}
        \frac{1}{\sqrt{n}}\Big[(1-t)\J_t-nt\Big]\overset{d}{\longrightarrow}\mathrm{Normal}(0,t)\,.
    \end{align*}
\end{fact}

This suggests that $([(1-t)J_t-nt]/\sqrt{n})_{t\in[0,1)}$ converges to some Gaussian process. Proving that this process indeed converges, and characterizing its distribution remains open.

\begin{conj}
    As $n$ goes to infinity we have that
    \begin{align*}
        \left(\frac{1}{\sqrt{n}}\Big[(1-t)\J_t-nt\Big]\right)_{t\in[0,1)}\overset{d}{\longrightarrow}(G_t)_{t\in[0,1)}\,,
    \end{align*}
    where $(G_t)_{t\in[0,1)}$ is a centred Gaussian and Markovian process.
\end{conj}

\subsection{Relation to random sorting networks}

From the results of Theorem~\ref{thm:MallowsGraph}, we know that any smooth Mallows process $\M=\Mt$ transitions using transpositions. This means that for any $\Mt$, there exists $\tau_1,\ldots,\tau_N\in\T$ where $N=\binom{n}{2}$ such that $\M_{T_k}=\tau_1\cdot\ldots\cdot\tau_k$, which implies that $\tau_1\cdot\ldots\cdot\tau_N=(n,n-1,\ldots,1)$. This naturally suggests an investigation of the relation between Mallows processes and sorting networks.

\textit{Sorting networks} are sequences $\tau_1,\ldots,\tau_N\in\Ta$ such that $\tau_1\cdot\ldots\cdot\tau_N=(n,n-1,\ldots,1)$. They were first considered by Stanley~\cite{stanley1984number}, who managed to give an exact formula for the number of such sorting networks of size $n$. Since then, further work on proving this formula through diverse methods was developed~\cite{edelman1987balanced,lascoux1982structure}, and these works recently led several authors~\cite{angel2019local,angel2007random,dauvergne2021archimedean,dauvergne2020circular} to study properties of uniformly sampled \textit{random sorting networks}. In these works, various results of random sorting networks are proven, such as the asymptotic distribution of $\tau_1$, a convergence for the scaled swap process, and properties of the trajectories $(\tau_1\cdot\ldots\cdot\tau_k(i))_{0\leq k\leq N}$ as $k$ varies from $0$ to $N$ and for a fixed $i\in[n]$. Mallows processes do not directly correspond to sorting networks, since they might use transpositions in $\T\setminus\Ta$, but many of the properties of the sorting networks that have previously been studied have natural analogues for Mallows processes and it would be natural to investigate these in more detail.

\section*{Acknowledgements}

BC wishes to thank his supervisor, Louigi Addario-Berry, for his help with the general structure and presentation of this paper, along with Orph\'ee Collin for early discussions on the problem. During the preparation of this research, BC was supported by an ISM scholarship.

\medskip
\bibliographystyle{siam}
\bibliography{main}

\end{document}